\documentclass[a4paper, reqno, 12pt]{amsart}
\usepackage[utf8]{inputenc}
\usepackage[T1, T2A]{fontenc}
\usepackage[english]{babel}
\usepackage{amsmath, amssymb, amsfonts, amsthm, amscd, mathrsfs}
\usepackage{enumitem}
\usepackage[hidelinks]{hyperref}
\usepackage[usenames]{color}
\usepackage{longtable}
\setlist[enumerate]{label = (\alph*), ref=(\text{\alph*)}}
\setlist[itemize]{nolistsep}

\usepackage{epsfig}
\usepackage{graphicx}
\usepackage[cmtip,arrow]{xy}
\usepackage{bm}
\usepackage{tikz}
\usetikzlibrary{calc}
\tikzstyle{precornfill} = [fill=black!10, draw = black]
\tikzstyle{cornfill} = [fill=black!30, draw = black]

\renewcommand{\phi}{\varphi}
\renewcommand{\ge}{\geqslant}
\renewcommand{\le}{\leqslant}

\newcommand{\KK}{\mathbb{K}}
\renewcommand{\AA}{\mathbb{A}}
\newcommand{\ZZ}{\mathbb{Z}}

\newcommand{\GG}{\mathbb{G}}
\newcommand{\PP}{\mathbb{P}}
\newcommand{\mm}{\mathfrak{m}}
\newcommand{\NN}{\mathbb{Z}_{>0}}
\newcommand{\Zgezero}{\mathbb{Z}_{\geqslant 0}}
\newcommand{\pa}{\partial}

\newcommand{\Xsing}{X^{\mathrm{sing}}}
\newcommand{\Zsing}{Z^{\mathrm{sing}}}
\newcommand{\lecur}{\preccurlyeq}

\DeclareMathOperator{\GL}{GL}

\DeclareMathOperator{\codim}{codim}
\DeclareMathOperator{\Lie}{Lie}
\DeclareMathOperator{\Mat}{Mat}
\DeclareMathOperator{\Soc}{Soc}
\DeclareMathOperator{\corn}{corn}
\DeclareMathOperator{\precorn}{precorn}

\theoremstyle{plain}
\newtheorem{lemma}{Lemma}
\newtheorem{proposition}{Proposition}
\newtheorem{theorem}{Theorem}
\newtheorem{corollary}{Corollary}
\theoremstyle{definition}
\newtheorem{definition}{Definition}
\newtheorem{example}{Example}

\newtheorem{construction}{Construction}

\theoremstyle{remark}
\newtheorem{remark}{Remark}

\begin{document}

\title[On normality of projective hypersurfaces with an additive action]{On normality of projective hypersurfaces \\ with an additive action}
\author{Ivan Arzhantsev}
\address{HSE University, Faculty of Computer Science, Pokrovsky Boulvard 11, Moscow, 109028 Russia}
\email{arjantsev@hse.ru}

\author{Ivan Beldiev}
\address{HSE University, Faculty of Computer Science, Pokrovsky Boulvard 11, Moscow, 109028 Russia}
\email{ivbeldiev@gmail.com}

\author{Yulia Zaitseva}
\address{HSE University, Faculty of Computer Science, Pokrovsky Boulvard 11, Moscow, 109028 Russia}
\email{yuliazaitseva@gmail.com}

\thanks{Supported by the Russian Science Foundation grant 23-21-00472}



\begin{abstract}
We study projective hypersurfaces $X$ admitting an induced additive action, i.e., an effective action ${\GG_a^m\times X\to X}$ of the vector group $\GG_a^m$ with an open orbit that can be extended to an action on the ambient projective space. A criterion for normality of such a hypersurface $X$ is given. Also, we prove that for any projective hypersurface $Z$ there exists a hypersurface $X$ with an induced additive action such that the complement to the open $\GG_a^m$-orbit in $X$ is a projective cone over $Z$. We introduce a construction that produces non-degenerate hypersurfaces with induced additive action from Young diagrams and study the properties of the hypersurfaces obtained in this way.
\end{abstract}

\maketitle

\section{Introduction}

We assume that the base field $\KK$ is an algebraically closed field of characteristic zero. All algebraic varieties we consider are algebraic varieties over $\KK$. For a positive integer $m$, we denote by~$\GG_a^m$ the algebraic group $(\KK^m, +)$.

An \emph{additive action} on an algebraic variety $X$ is an effective regular action of the group~$\GG_a^m$ on $X$ with an open orbit. We are mainly interested in the case when $X\subseteq \PP^n$ is a projective hypersurface and an additive action $\GG_a^m\times X\to X$ can be extended to a regular action of the group $\GG_a^m$ on the ambient projective space $\PP^n$. Such actions on $X$ are called \emph{induced additive actions}. Note that $m=n-1$ for dimension reasons.

Although not all additive actions on projective hypersurfaces are induced (see \cite[Example~2.2]{AZa}), the property of being induced is not very restrictive. For example, let $X\subseteq \PP^n$ be a linearly normal subvariety. This means that $X$ is not contained in a hyperplane and cannot be obtained as a linear projection of a subvariety from a bigger projective space, or, equivalently,  the restriction map $H^0(\PP^n, \mathcal O(1))\to H^0(X, \mathcal O(1))$ is surjective. Then any additive action on $X$ is induced; see \cite[Section 2]{AP}.

In \cite{HT}, Hassett and Tschinkel establish a natural bijection between additive actions on~$\PP^n$ and local finite-dimensional commutative associative unital algebras of dimension~$n+1$. In \cite[Section 1.5]{AZa}, the authors prove a generalized version of the Hassett-Tschinkel correspondence. It says that there is, up to natural equivalences, a bijection between the following objects:

\begin{enumerate}
    \item induced additive actions on projective hypersurfaces in $\PP^n$ different from a hyperplane;
    \item pairs $(A, U)$, where $A$ is a local commutative associative unital algebra over $\KK$ of dimension $n+1$ with the maximal ideal $\mm$ and $U\subseteq \mm$ is a hyperplane generating the algebra $A$. Such pairs $(A,U)$ are called \emph{$H$-pairs}.
\end{enumerate}

It is natural to study geometric properties of projective hypersurfaces admitting an induced additive action. For example, one may ask when such a hypersurface is smooth. It turns out that the only smooth hypersurfaces admitting an induced additive action are hyperplanes and non-degenerate quadrics; see \cite[Proposition~4]{AP}.

The question of normality of a hypersurface with an induced additive action is more delicate; see, for instance, examples in \cite[Section~4]{Bel}. In Section~3, we give a criterion for normality of such a hypersurface in terms of the first two homogeneous components $f_d$ and $f_{d-1}$ of the polynomial $f$ defining the hypersurface~$X$ of degree~$d$; see Proposition~\ref{propnorm}. In particular, it implies that $X$ is normal, provided the polynomial $f_d$ is square-free; see Corollary~\ref{cornorm}.

Let us denote by $X_0$ the complement in $X$ to the open $\GG_a^m$-orbit and call it the boundary of $X$. It is a hypersurface in a hyperplane of the ambient projective space $\PP^n$. In Section~4, we prove that the boundary $X_0$ is either a projective subspace or a projective cone over a hypersurface $Z_0$ of smaller dimension, see Proposition~\ref{boudcone_prop}. In the second case, the variety $X_0$ is always singular. Moreover, in Theorem~\ref{thb} we prove that for any hypersurface $Z_0$, there is a hypersurface $X$ with an induced additive action whose boundary $X_0$ is a projective cone over $Z_0$. To show this, we give an explicit construction of a suitable $H$-pair $(A, U)$. It implies that in certain dimensions, there are infinitely many pairwise non-isomorphic normal hypersurfaces admitting an induced additive action.

In Section~5, we introduce the class of projective hypersurfaces coming from $H$-pairs $(A_{\Lambda,B}, U_{\Lambda,B})$ associated with Young diagrams~$\Lambda$. We give a criterion for an algebra~$A_{\Lambda,B}$ to be Gorenstein via a system of linear equations. It is a system on the variables corresponding to precorner cells of the Young diagram~$\Lambda$, and we prove that the number of precorner cells is at least the dimension of~$\Lambda$ in non-exceptional cases; see Proposition~\ref{precorn_prop}. If $A_{\Lambda,B}$ is not Gorenstein, one can apply the reduction procedure and obtain the $H$-pair $(\widetilde{A_{\Lambda,B}}, \widetilde{U_{\Lambda,B}})$ with Gorenstein algebra, which corresponds to a non-degenerate projective hypersurface. The considered family of hypersurfaces contains the hypersurfaces used in the proof of the main result of Section 4 and all the Gorenstein algebras of dimension at most~6. We illustrate the result obtained in Section~3 and give precise formulas for the polynomials $f_d$ and~$f_{d-1}$, which determine the normality of the hypersurface; see Proposition~\ref{propYD}. It implies the combinatorial sufficient condition for normality of such hypersurfaces; see Proposition~\ref{prYD}.

It is proved in~\cite[Corollary~5.2]{AS} that a hypersurface $X\subseteq\PP^n$ admitting an induced additive action has degree at most $n$. In Proposition~\ref{pp}, we show that such a hypersurface of degree $n$ is unique. We find an explicit form of the equation of this hypersurface and prove that it is normal if and only if $n\le 2$.

\section{Preliminaries}

We need some basic definitions and results on local finite-dimensional algebras. Recall that an algebra is called \emph{local} if it has a unique maximal ideal $\mm$. All algebras are assumed to be commutative, associative, and unital.

\begin{lemma} \cite[Lemma 1.2]{AZa}
    A finite-dimensional algebra $A$ is local if and only if $A$ is the direct sum of its subspaces $\KK \oplus \mm$, where $\mm$ is the ideal consisting of all nilpotent elements of $A$.
\end{lemma}

\begin{definition}
    The \emph{socle} of a local algebra $A$ with the maximal ideal $\mm$ is the ideal $$\Soc A = \{a\in A\mid a \mm = 0\}.$$ A local finite-dimensional algebra $A$ is called \emph{Gorenstein} if $\dim \Soc A = 1$.
\end{definition}

If $d$ is the maximal number such that $\mm^d\ne 0$, then $\mm^d\subseteq \Soc A$. However, this inclusion can be strict. So, $A$ is Gorenstein if and only if $\dim \mm^d = 1$ and $\Soc A = \mm^d$.

Next, we give a formal definition of equivalence of induced additive actions on projective hypersurfaces.

\begin{definition}
    Two induced additive actions $\alpha_i\colon \GG_a^m \times X_i\to X_i$, $X_i\subseteq \PP^n$, $i = 1, 2$, are called \emph{equivalent} if there exists an automorphism of an algebraic group $\phi\colon \GG_a^m \to \GG_a^m$ and an automorphism $\psi \colon \PP^n \to \PP^n$ such that $\psi(X_1) = X_2$ and $\psi\circ\alpha_1=\alpha_2\circ(\phi\times\psi).$
\end{definition}

Let us give a definition of an $H$-pair.

\begin{definition}
\label{pdpd}
An \emph{$H$-pair} is a pair $(A,U)$, where $A$ is a local finite-dimensional algebra with the maximal ideal $\mm$ and $U\subseteq \mm$ is a hyperplane generating $A$ as a unital algebra.
\end{definition}

One can define equivalence of $H$-pairs as follows.

\begin{definition}
    Two pairs $(A_1, U_1)$ and $(A_2, U_2)$ are called \emph{equivalent} if there exists an isomorphism of algebras $\phi\colon A_1 \to A_2$ such that $\phi(U_1) = U_2$.
\end{definition}

Now, we give the precise statement of the generalized version of the Hassett-Tschinkel correspondence.

\begin{theorem}\cite[Theorem 2.6]{AZa}
\label{prth}
    Suppose that $n\in \NN$. There is a one-to-one correspondence between the following objects:
    \begin{enumerate}
        \item induced additive actions on hypersurfaces in $\PP^n$ not contained in any hyperplane;
        \item pairs $(A,U)$, where $A$ is a local commutative associative unital algebra of dimension~$n+1$ with the maximal ideal $\mm$ and $U\subseteq \mm$ is a hyperplane generating the algebra~$A$.
    \end{enumerate}
    This correspondence is considered up to equivalences from Definitions 2 and 4.
\end{theorem}

The construction of this correspondence is done as follows. For an $H$-pair $(A, U)$, denote by $p\colon A \setminus \{0\} \to \PP(A)\cong \PP^n$ the canonical projection. Then we define
$$X = p(\overline{\KK^{\times}\exp U}),$$
i.e., it is the projectivization of the Zariski closure of the subset $\KK^{\times}\exp U\subseteq A\setminus\{0\}$. Since $A$ is commutative, the algebraic group $\exp U$ can be identified with $\GG_a^{n-1}$; hence the multiplication by elements of $\exp U$ defines an action of $\GG_a^{n-1}$ on $\PP(A)$. It is easy to see that $X$ is preserved under this multiplication, so this defines an induced additive action of $\GG_a^{n-1}$ on $X \subseteq \PP(A) \cong \PP^n$.

Conversely, an induced additive action of $\GG_a^{n-1}$ on a hypersurface $X\subseteq \PP^n = \PP(V)$, where $\dim V = n+1$, can be lifted to a linear action of $\GG_a^{n-1}$ on $V$, which gives us a faithful representation $\rho \colon \GG_a^{n-1} \to \GL_{n+1}(\KK)$. Let $U$ be the vector space $d\rho(\mathfrak g_a^{n-1})$ and define $A$ as the unital subalgebra of $\Mat_{n+1}(\KK)$ generated by $U$; here
$$
d\rho\colon \mathfrak g_a^{n-1} = \Lie(\GG_a^{n-1})\to \Mat_{n+1}(\KK)
$$
is the differential of the map $\rho$. It can be checked that $(A, U)$ is an $H$-pair. One can find the details in \cite[Theorem~1.38]{AZa}.

Given an $H$-pair $(A, U)$, the equation of the corresponding hypersurface is computed as follows (for more details, see \cite[Chapter 2.2]{AZa}). Denote by $\pi$ the canonical projection $\pi\colon U \to U/\mm$. Let $d$ be the greatest positive integer such that $\mm^d \nsubseteq U$. The corresponding projective hypersurface is given by the homogeneous equation
\begin{equation}
\label{hyp_eq}
z_0^d\pi\left(\ln\Bigl(1 + \frac{z}{z_0}\Bigr)\right) = 0
\end{equation}
for $z_0 + z\in A = \KK \oplus \mm$, $z_0\in\KK$, $z\in \mm$. This hypersurface is irreducible and has degree~$d$.

\begin{definition}
   The hypersurface $X$ given by the equation $f(z_0, z_1, \ldots, z_n) = 0$, where $f$ is a homogeneous polynomial, is called \emph{non-degenerate} if one of the following equivalent conditions holds:
    \begin{enumerate}
    \item there exists no linear transform of variables such that the number of variables in $f$ after this transform becomes less than $n+1$;
    \item the hypersurface $X$ is not a projective cone over a hypersurface $Z\subseteq \PP^k$ in a projective subspace $\PP^k \subseteq \PP^n$ for some $k < n$.
    \end{enumerate}
\end{definition}
The following theorem is proved in \cite[Theorem 2.30]{AZa}.

\begin{theorem}
\label{tgor}
Induced additive actions on non-degenerate hypersurfaces of degree $d$ in $\PP^n$ are in one-to-one correspondence with $H$-pairs $(A, U)$, where $A$ is a Gorenstein local algebra of dimension~$n+1$ with the socle $\mm^d$ and $\mm = U \oplus \mm^d$. 
\end{theorem}

In \cite{AP,AZa}, the authors consider the procedure of reduction of an induced additive action. Namely, consider an $H$-pair $(A,U)$ corresponding to an induced additive action on a hypersurface $X$. Let $J \subseteq A$ be an ideal of dimension $n - k$ contained in $U$.

\begin{proposition} \cite[Proposition 2.20 and Corollary 2.23]{AZa}
\label{prop_reduct}
The pair $(A/J, U/J)$ corresponds to an induced additive action on a projective hypersurface $Z\subseteq \PP^{k}$, and $X$ is the projective cone over $Z$, i.e., for some choice of coordinates in $\PP^{n}$ and $\PP^{k}$, the equations of the hypersurfaces~$X$ and $Z$ are the same. Moreover, if $J$ is the maximal (with respect to inclusion) ideal of~$A$ contained in~$U$, then $Z$ is a non-degenerate hypersurface in $\PP^k$.
\end{proposition}

Finally, we are going to use the following criterion for normality of an affine or projective hypersurface. Denote by $\Xsing$ the singular locus of an algebraic variety $X$. It is well known that $\codim_X \Xsing \ge 2$ if $X$ is normal. For an arbitrary algebraic variety, the converse does not hold. However, this is true if $X$ is a hypersurface.

\begin{proposition} \cite[Section 5.1, Chapter 2]{Sha}
\label{snor}
A hypersurface $X$ is normal if and only if $\codim_X \Xsing \ge 2$.
\end{proposition}

\section{A criterion for normality}
\label{sn}

Let $X\subseteq\PP^n$ be a projective hypersurface of degree $d$ admitting an additive action. Then $X$ is given by
$$
f(z_0, z_1, \ldots, z_n) = z_0^d\pi\left(\ln\Bigl(1 + \frac{z}{z_0}\Bigr)\right) = 0,
$$
see equation~\eqref{hyp_eq}. Here $f$ is a homogeneous polynomial in $z_0, z_1, \ldots, z_n$ of degree $d$ without the term $z_0^d$; let us write the decomposition
\begin{equation}
\label{eqsumf}
f(z_0, z_1, \ldots,z_n) = \sum\limits_{k=1}^d z_0^{d-k} f_k(z_1, \ldots, z_n),
\end{equation}
where $f_k$ is a homogeneous polynomial in $z_1, z_2, \ldots, z_n$ of degree $k$. Moreover, the open orbit of the additive action is the set of all points $[z_0:z_1:\ldots: z_n]\in X$ such that $z_0\ne 0$. In particular, the complement to the open orbit is given by the following equation in $\PP^{n-1} = \{z_0=0\}$:
$$
f_d(z_1, z_2, \ldots, z_n) = 0.
$$
We call this complement the \emph{boundary} of $X$ and denote it by $X_0$. This is a hypersurface in $\PP^{n-1}$, so all its irreducible components have dimension~$n-2$.

The following proposition shows that normality of $X$ depends only on the polynomials $f_d$ and $f_{d-1}$. Let $f_d=p_1^{a_1}\ldots p_s^{a_s}$, where $p_i$ are pairwise distinct irreducible polynomials and $a_i\in\ZZ_{>0}$. Denote by $\widetilde{f_d}$ the polynomial $f_d$ divided by $p_1\ldots p_s$.

\begin{proposition}
\label{propnorm}
The hypersurface $X$ is normal if and only if the polynomials $\widetilde{f_d}$ and $f_{d-1}$ are coprime.
\end{proposition}

\begin{proof}
We use Proposition~\ref{snor}. All points of the open orbit of the additive action on $X$ are smooth, so $\Xsing \subseteq X_0 = X \cap \{z_0 = 0\}$.
Since $X_0$ has codimension~$1$ in $X$, the hypersurface $X$ is not normal if and only if $\Xsing$ has an irreducible component of the same dimension as~$X_0$.

It is known that
$$
\Xsing = \left\{\frac{\pa f}{\pa z_0} = \frac{\pa f}{\pa z_1} = \ldots = \frac{\pa f}{\pa z_n} = f = 0\right\} \subseteq \PP^n.
$$
Since $\Xsing \subseteq \{z_0 = 0\}$ and equation~\eqref{eqsumf} holds, we have
\begin{gather*}
\Xsing = \left\{\frac{\pa f}{\pa z_0} = 0\right\} \cap \left\{\frac{\pa f}{\pa z_1} = \ldots = \frac{\pa f}{\pa z_n} = f = 0\right\} = \\
= \left\{f_{d-1} = 0 \right\} \cap \left\{\frac{\pa f_d}{\pa z_1} = \ldots = \frac{\pa f_d}{\pa z_n} = f_d = 0\right\} \subseteq \{z_0 = 0\}.
\end{gather*}

Recall that $f_d=p_1^{a_1}\ldots p_s^{a_s}$, where $p_i$ are pairwise distinct irreducible polynomials and $a_i\in\ZZ_{>0}$, so any irreducible component of $\Xsing$ of dimension $n-2$ is $\{p_i = 0\}$ for some $1 \le i \le s$.

Notice that irreducible components of $\left\{\frac{\pa f_d}{\pa z_1} = \ldots = \frac{\pa f_d}{\pa z_n} = f_d = 0\right\}$ of dimension $n-2$ are given by equations $\{p_i = 0\}$, where $a_i \ge 2$. Indeed, for any $1 \le k \le n$, we have
$$
\frac{\pa f_d}{\pa z_k} = \sum\limits_{j = 0}^s a_j p_j^{a_j - 1}\frac{\pa p_j}{\pa z_k} p_1^{a_1} \ldots p_{j-1}^{a_{j-1}}p_{j+1}^{a_{j+1}}\ldots p_s^{a_s},
$$
where all summands are divisible by $p_i$ if $a_i \ge 2$, and all summands except one are divisible by $p_i$ if $a_i = 1$.

So irreducible components of $\Xsing$ of dimension $n-2$ are the hypersurfaces $\{p_i = 0\}$, where $a_i \ge 2$ and $f_{d-1}$ is divisible by $p_i$. Thus, $X$ is normal if and only if $f_{d-1}$ and $\widetilde f_d = p_1^{a_1 - 1}\ldots p_s^{a_s - 1}$ are coprime.
\end{proof}

\begin{corollary}
\label{cornorm}
If the polynomial $f_d$ is square-free, then the hypersurface $X$ is normal.
\end{corollary}

We need the following normality criterion for projective cones below.

\begin{lemma}
\label{l1}
Let $Z$ be a hypersurface in $\PP^k$ of dimension at least one and degree at least two. Consider $\PP^k$ as a subspace in $\PP^n$ with $n>k$. Then the projective cone $X$ over $Z$ in $\PP^n$ is normal if and only if $Z$ is normal. Moreover, the variety $X$ is singular.
\end{lemma}

\begin{proof}
A point $z = [z_0:\ldots:z_k] \in Z = \{g(z_0,\ldots,z_k) = 0\} \subseteq \PP^k$ is singular if and only if $\frac{\pa g}{\pa z_i}(z) = 0$ for any $0 \le i \le k$. Let $X = \{g(z_0, \ldots, z_k) = 0\} \subseteq \PP^n$ be the projective cone over~$Z$. Since $\frac{\pa g}{\pa z_{k+j}} = 0$, $1 \le j \le n-k$, the singular locus $\Xsing$ of $X$ is divided into two subsets:
\smallskip
\begin{enumerate}
    \item $X_1 = \{[z_0:\ldots:z_n], \; \text{ where } [z_0:\ldots:z_k] \in \Zsing\}$;
    \item $X_2 = \{[0:\ldots:0:z_{k+1}:\ldots:z_n]\} \cong \PP^{n-k-1}$.
\end{enumerate}
\smallskip
We again use that a projective hypersurface is normal if and only if the codimension of the singular locus is at least~2. Notice that
\[\begin{aligned}
\codim_X X_1 &= \codim_Z \Zsing;\\
\codim_X X_2 &= (n-1) - (n-k-1) = k \ge 2.
\end{aligned}\]
So the condition $\codim_X \Xsing \ge 2$ is equivalent to $\codim_Z \Zsing \ge 2$. Moreover, we have $\Xsing\supseteq\PP^{n-k-1}\ne\varnothing$.
\end{proof}

\section{Non-degenerate hypersurfaces with prescribed boundary}
\label{sp}

Let a hypersurface $X \subseteq \PP^n$ of degree $d \ge 2$ admit an additive action. We let $(A, U)$ be the corresponding $H$-pair, $\mm$ be the maximal ideal of~$A$,
$$
\pi\colon \mm \to \mm/U \cong \KK,\quad \text{and} \quad p\colon A\setminus\{0\} \to \PP(A)
$$
be the canonical projection. We identify $\PP^n$ with $\PP(A)$, and the boundary $X_0$ is contained in $\PP(\mm)\cong\PP^{n-1}$. According to \cite[Corollary~2.18]{AZa}, we have
\[X_0 = p(\{z \in \mm \mid z^d \in U\}),\]
i.e., $X_0 \subseteq \PP(\mm)$ is given by the equation $\pi(z^d) = 0$, $z \in \mm$.

\begin{lemma}\label{lbeq}
For any decomposition $\mm = L \oplus \mm^2$ into a direct sum of subspaces, the boundary~$X_0$ is given in $\PP(\mm)$ by the equation
\[\pi(z_L^d) = 0,\]
where $z = z_{L} + z_{\mm^2}$, $z_{L} \in L$, $z_{\mm^2} \in \mm^2$. In particular, the equation of $X_0$ depends on at most $\dim \mm/\mm^2$ coordinates.
\end{lemma}

\begin{proof}
According to~\cite[Theorem 5.1]{AS}, the degree $d$ of $X$ is the maximal number such that $\mm^d \nsubseteq U$. Then $\pi(\mm^{d+1}) \subseteq \pi(U) = 0$, so $\pi(z^d) = \pi((z_{L}+z_{\mm^2})^d) = \pi(z_{L}^d)$. So $\pi(z_{L}^d) = 0$ is the equation of $X_0$ and it depends on at most $\dim L = \dim \mm - \dim \mm^2$ coordinates.
\end{proof}

\begin{proposition}
\label{boudcone_prop}
The boundary $X_0$ is a degenerate projective hypersurface in $\PP(\mm)$. In particular, $X_0$ is either singular or a hyperplane.
\end{proposition}

\begin{proof}
If $X_0$ is non-degenerate, then $\dim \mm - \dim \mm^2 = n$, so $\mm^2 = 0$ and there is no $H$-pair $(A,U)$ corresponding to such an algebra $A$. The additive action
$\GG_a^n\times\PP^n\to\PP^n$ corresponding to the algebra $A$ is
$$
(a_1,\ldots,a_n)\circ [z_0:z_1:\ldots:z_n]=[z_0:z_1+a_1z_0:\ldots:z_n+a_nz_0],
$$
and the closure $X$ of an orbit of a subgroup of codimension one in $\GG_a^n$ is a hyperplane in $\PP^n$. The boundary $X_0$ is the intersection of $\PP(\mm)$ with the hyperplane $X$, so it is a hyperplane in $\PP(\mm)$. But a hyperplane in $\PP(\mm)$ is a degenerate hypersurface.

The last assertion follows from Lemma~\ref{l1}.
\end{proof}

The following result shows that any non-degenerate projective hypersurface up to taking projective cone can be realized as the boundary $X_0$.

\begin{theorem}
\label{thb}
For any non-degenerate hypersurface $Z_0\subseteq\PP^{k-1}$, there exist an integer $n$ and a non-degenerate hypersurface $X\subseteq\PP^n$ with an additive action such that the boundary $X_0$ is a projective cone over~$Z_0$ in $\PP^{n-1}$.
\end{theorem}

\begin{proof}
Let $Z_0 = \{g = 0\} \subseteq \PP^{k-1}$, where $g$ is a non-zero homogeneous polynomial of degree~$d$ in variables $z_1, z_2,\ldots, z_k$. Let us write $g$ in the following way:
    $$g(z_1, z_2,\ldots, z_k) = \sum\limits_{|\lambda| = d} a_{\lambda}z^{\lambda},$$
where the summation is taken over all $k$-tuples $\lambda = (\lambda_1, \lambda_2, \ldots, \lambda_k)$ of non-negative integers, $|\lambda| = \lambda_1 + \lambda_2 + \ldots + \lambda_k$, and $z^\lambda = z_1^{\lambda_1}\ldots z_k^{\lambda_k}$. Since $g$ is non-zero, there is at least one~$\lambda_0$ with $a_{\lambda_0} \ne 0$.

We are going to construct an $H$-pair $(A, U)$ such that the intersection of the corresponding hypersurface $X$ with the hyperplane $\{z_0 = 0\}$ is given by the equation $g = 0$. We define the algebra $A$ by generators $x_1, x_2, \ldots, x_k$ and relations
\[\begin{aligned}
x^{\lambda} &= 0 \hspace{1cm} \text{for all $\lambda$ with }\hspace{0.1cm} |\lambda| = d+1, \\
x^{\lambda} &= b_\lambda x^{\lambda_0} \hspace{0.2cm} \text{for all $\lambda$ with}\hspace{0.1cm}|\lambda| = d, \; \lambda\ne \lambda_0,
\end{aligned}\]
where $b_\lambda\in\KK$ are some coefficients that we yet need to define.
The maximal ideal $\mm$ of the algebra $A$ has the monomial basis consisting of all monomials in $x_i$ of degree less than $d$ and of one monomial $x^{\lambda_0}$ of degree $d$. We define $U$ as the linear span of all such monomials except $x^{\lambda_0}$.

By construction, $\mm = \langle x_1, \ldots, x_k\rangle \oplus \mm^2$ and $\mm = U \oplus \mm^d$. Let $X'\subseteq\PP(A)$ be the projective hypersurface corresponding to the pair $(A,U)$. According to Lemma~\ref{lbeq}, the boundary~$X'_0$ is given by the equation $\pi(z_L^d) = 0$, where $z_L = z_1x_1 + z_2x_2 + \ldots + z_kx_k$ and $\pi\colon \mm \to \mm^d = \langle x^{\lambda_0}\rangle$ is the projection of $\mm$ on $\mm^d$ along~$U$. Notice that
$$z_L^d = (z_1x_1 + z_2x_2 + \ldots + z_kx_k)^d = \sum\limits_{|\lambda| = d} c_{\lambda}z^{\lambda}x^{\lambda} = \big(\sum\limits_{|\lambda| = d} b_\lambda c_\lambda z^\lambda \big)x^{\lambda_0},$$
where $c_{\lambda} = \frac{(\lambda_1 + \lambda_2 + \ldots + \lambda_k)!}{\lambda_1!\lambda_2!\ldots\lambda_k!} \ne 0$ are the multinomial coefficients.
This implies that the boundary has the equation
$$\sum\limits_{|\lambda| = d} b_\lambda c_\lambda z^\lambda = 0.$$

Let us take $b_\lambda = \frac{a_\lambda c_{\lambda_0}}{c_\lambda a_{\lambda_0}}$. This agrees with the condition $b_{\lambda_0}=1$. Moreover, the left-hand side of the last equation becomes equal to $g$ multiplied by $\frac{c_{\lambda_0}}{a_{\lambda_0}}$.

So, we constructed the $H$-pair $(A,U)$ corresponding to a possibly degenerate hypersurface~$X'$ whose boundary~$X'_0$ is given by the equation $g = 0$. In order to make the hypersurface non-degenerate, we apply the reduction procedure to the $H$-pair $(A, U)$ and obtain a new $H$-pair $(A/J, U/J)$. 
Since reduction does not change the equation of the hypersurface, the $H$-pair $(A/J, U/J)$ corresponds to the desired non-degenerate hypersurface~$X\subseteq\PP^n$. Indeed, the boundaries $X_0 = X \cap \{z_0 = 0\}$, $X'_0 = X' \cap \{z_0 = 0\}$, and the hypersurface~$Z_0$ have the same equation $g = 0$, so $X_0$ is a projective cone over $Z_0$ since $Z_0$ is non-degenerate.
\end{proof}

\section{Algebras associated with Young diagrams}
\label{asYD}

In this section, we illustrate the result obtained above in the case of projective hypersurfaces coming from algebras associated with Young diagrams.

Let $k \in \NN$. We introduce some notation. Consider the product order $\lecur$ on $\Zgezero^k$, i.e., $\lambda \lecur \mu$ means that $\lambda_i \le \mu_i$ for any $1 \le i \le k$, where $\lambda,\mu \in \Zgezero^k$. Let $\Lambda$ be a Young diagram of dimension~$k$, i.e., a finite subset $\Lambda \subseteq \Zgezero^k$ such that for any $\mu \in \Lambda$ and any $\lambda\in \ZZ^k_{\geq 0}$ such that $\lambda \lecur \mu$ we have $\lambda \in \Lambda$. Let us assume that for any $1 \le i \le k$, there exists $\lambda \in \Lambda$ such that $\lambda_i \ne 0$.

\begin{definition}
An element $\mu$ of a Young diagram $\Lambda$ is a \emph{corner cell} if there is no $\mu' \in \Lambda$ with $\mu \prec \mu'$. In other words, corner cells are maximal elements in $\Lambda$ with respect to the partial order $\lecur$. Denote the set of corner cells of $\Lambda$ by $\corn(\Lambda)$.
\end{definition}

Any Young diagram is defined by the set of its corner cells.

\smallskip

Let $\KK[x] := \KK[x_1, \ldots, x_k]$ be a polynomial algebra. Recall that for any $\lambda \in \Zgezero^k$, we denote
$|\lambda| = \lambda_1 + \lambda_2 + \ldots + \lambda_k$ and $x^\lambda = x_1^{\lambda_1}\ldots x_k^{\lambda_k}$.

\begin{construction}
\label{constrYD_0}
Let $\Lambda$ be a Young diagram of dimension $k$. One can consider a finite-dimensional local algebra $A_\Lambda$ defined as the quotient algebra of $\KK[x_1, \ldots, x_k]$ by relations
\begin{equation} \label{eqYDrel0}
x^\lambda = 0 \;\;\text{ if } \lambda \notin \Lambda.
\end{equation}
\end{construction}

The algebra $A_\Lambda$ is monomial, i.e., the defining ideal of this algebra is generated by monomials. Moreover, any finite-dimensional monomial algebra has this form.

Notice that $\Soc A_\Lambda = \langle x^\mu \mid \mu \in \corn(\Lambda)\rangle$. So, the algebra $A_\Lambda$ is Gorenstein if and only if $\Lambda$ has a unique corner cell, i.e., $\Lambda$ is a parallelepiped.

\smallskip

Before we proceed, let us give the definitions of exceptional and non-exceptional Young diagrams.  Suppose that $\Lambda$ is a $k$-dimensional Young diagram. For $1 \le i \le k$, we call the diagram $\Lambda$ \emph{exceptional with respect to the $i$-th coordinate} if for any $\lambda \in \Lambda$ we have $\lambda_i = 0$ or $\lambda = e_i$. If this does not hold for all $1 \le i \le k$, the diagram $\Lambda$ is said to be \emph{non-exceptional}.

Now, we come to the following construction.

\begin{construction}
\label{constrYD}
Suppose that we have a Young diagram $\Lambda$ of dimension $k$ and a set of constants $B = \{b_\mu \in \KK, \; \mu \in \corn(\Lambda)\}$ such that at least one constant is non-zero. We construct an $H$-pair $(A_{\Lambda, B}, U_{\Lambda, B})$ using this data.

Define $A_{\Lambda,B}$ as the quotient algebra of $A_\Lambda$ by relations
\begin{equation} \label{eqYDrel}
b_\nu x^\mu = b_\mu x^\nu   \;\;\text{ if } \mu,\nu \in \corn(\Lambda),
\end{equation}
and $U_{\Lambda,B}$ as the linear span of $x^\lambda$, $\lambda \in \Lambda \setminus (\corn(\Lambda) \cup \{0\})$. According to~\eqref{eqYDrel}, there exists a monomial $x^{\corn} \in A$ such that
\begin{equation}
x^\mu = b_\mu x^{\corn}   \;\;\text{ for any } \mu \in \corn(\Lambda)
\end{equation}
(note that the coefficient of the monomial $x^{corn}$ is not necessarily equal to $1$).

Notice that the algebra $A_{\Lambda, B}$ associated with a diagram $\Lambda$ exceptional with respect to the $i$-th coordinate is isomorphic to the algebra associated with the same diagram without the cell~$e_i$.

Assume that the diagram $\Lambda$ is non-exceptional. Then, the algebra $A_{\Lambda,B}$ is local with the maximal ideal $\mm = U_{\Lambda,B} \oplus \langle x^{\corn}\rangle$; the subspace~$U_{\Lambda,B}$ has codimension one in~$\mm$ and generates~$A_{\Lambda,B}$. So, $(A_{\Lambda,B}, U_{\Lambda,B})$ is an $H$-pair.

Let $n$ be the number of cells in $\Lambda$ without corner cells. Clearly, $\dim A_{\Lambda,B} = n+1$.

\begin{definition}
We call the $H$-pair $(A_{\Lambda,B}, U_{\Lambda,B})$ the \emph{$H$-pair associated with $(\Lambda, B)$}.
\end{definition}
\end{construction}

\begin{example}
\label{YD1ex}
Consider the Young diagram $\Lambda$ in Figure~\ref{YD1fig} with $b_\mu = 1$ for both $\mu \in \corn(\Lambda)$. Then
\[A_{\Lambda,B} = \KK[x, y]/(x^4,\, y^3,\, x^2y^2,\, x^3y - xy^2).\]
Notice that it is not Gorenstein. Indeed, $x^{\corn} \in \Soc A$, and for $a = x^2y - y^2$, we have $ax = x^3y - xy^2 = 0$ and $ay = 0 - 0 = 0$, so $a \in \Soc A$ as well.
\begin{figure}[h]
\begin{tikzpicture}[x=0.75pt,y=0.75pt,yscale=-1,xscale=1]
\def\d{15}
\coordinate (e1) at (\d,0);
\coordinate (e2) at (0,-\d);
\foreach \x in {0,..., 3}
    \foreach \y in {0,1}
        \draw ($2*\x*(e1)+2*\y*(e2)$) + (-\d,-\d) rectangle ++(\d,\d);
\foreach \x/\y in {3/0, 2/1, 0/2}
    \filldraw[precornfill] ($2*\x*(e1)+2*\y*(e2)$) + (-\d,-\d) rectangle ++(\d,\d);
\foreach \x/\y in {3/1, 1/2}
    \filldraw[cornfill] ($2*\x*(e1)+2*\y*(e2)$) + (-\d,-\d) rectangle ++(\d,\d);
\draw ($2*0*(e1)$) node{$1^{\phantom{1}}\!\!$};
\draw ($2*1*(e1)$) node{$x^{\phantom{1}}\!\!$};
\draw ($2*1*(e2)$) node{$y^{\phantom{1}}\!\!$};
\draw ($2*2*(e1)$) node{$x^2$};
\draw ($2*2*(e2)$) node{$y^2$};
\draw ($2*3*(e1)$) node{$x^3$};
\draw ($2*1*(e1)+2*1*(e2)$) node{$xy^{\phantom{1}}\!\!$};
\draw ($2*2*(e1)+2*1*(e2)$) node{$x^2y$};
\draw ($2*1*(e1)+2*2*(e2)$) node{$xy^2$};
\draw ($2*3*(e1)+2*1*(e2)$) node{$x^3y$};
\end{tikzpicture}
\caption{The Young diagram in Example~\ref{YD1ex}.}
\label{YD1fig}
\end{figure}
\end{example}

\begin{construction}
As it is shown in Example~\ref{YD1ex}, the algebra $A_{\Lambda,B}$ still may not be Gorenstein for some data $(\Lambda, B)$. Let $J$ be the maximal (with respect to inclusion) ideal of~$A_{\Lambda,B}$ contained in~$U_{\Lambda,B}$. By Proposition~\ref{prop_reduct}, the pair $(\widetilde{A_{\Lambda,B}}, \widetilde{U_{\Lambda,B}}) = (A_{\Lambda,B}/J, U_{\Lambda,B}/J)$ is an $H$-pair, and the hypersurface that corresponds to $(A_{\Lambda,B}, U_{\Lambda,B})$ is the projective cone over a non-degenerate hypersurface that corresponds to $(\widetilde{A_{\Lambda,B}}, \widetilde{U_{\Lambda,B}})$. According to Theorem~\ref{tgor}, the algebra $\widetilde{A_{\Lambda,B}}$ is Gorenstein.

\begin{definition}
The \emph{reduced $H$-pair associated with $(\Lambda, B)$} is the above $H$-pair $(\widetilde{A_{\Lambda,B}}, \widetilde{U_{\Lambda,B}})$.
\end{definition}
\end{construction}

Notice that the $H$-pairs that appear in the proof of Theorem~\ref{thb} are the algebras $\widetilde{A_{\Lambda,B}}$, where $\Lambda = \{\lambda \mid |\lambda| \le d\}$.

We denote the coordinates of an element $x$ in the algebra $A_{\Lambda,B}$ of dimension $n+1$ as follows:
\begin{equation}
\label{eqYDcoord}
x = \sum\limits_{\lambda \in \Lambda \setminus \corn(\Lambda)} z_\lambda x^\lambda + z_{\corn} x^{\corn} \in A_{\Lambda,B}.
\end{equation}
For shortness, we also define $z_0 = z_{(0, \ldots, 0)}$, $z_i = z_\lambda$ for $\lambda = e_i := (0, \ldots, \underset{i}{1}, \ldots, 0)$, $1 \le i \le k$, and let $z^\lambda = z_1^{\lambda_1}\ldots z_k^{\lambda_k}$ for $\lambda = (\lambda_1, \ldots,\lambda_k)$. The same notation is used for homogeneous coordinates in the projective space $\PP^n = \PP(A_{\Lambda,B})$.

\begin{remark}
\label{precorn_rem}
Let us discuss the conditions for the algebra $A_{\Lambda,B}$ to be Gorenstein. We call a non-corner cell $\lambda \in \Lambda$ a \emph{precorner cell} if $\lambda + e_i \notin \Lambda$ or $\lambda + e_i \in \corn(\Lambda)$ for any $1 \le i \le k$. The set of precorner cells is denoted by $\precorn(\Lambda)$.
Consider an element $x$ as in equation~\eqref{eqYDcoord}. Then $x \in \Soc A_{\Lambda,B}$ if and only if $xx_i = 0$ for any $1 \le i \le k$. Notice that
\begin{align*}
xx_i = \!\!\sum_{\lambda \in \Lambda \setminus \corn(\Lambda)}\!\!\! z_\lambda x^{\lambda+e_i} &=
\!\!\sum_{\substack{\lambda \in \Lambda\setminus \corn(\Lambda)\\\lambda + e_i \in \Lambda \setminus \corn(\Lambda)}}\!\!\! z_\lambda x^{\lambda+e_i} +
\!\!\sum_{\substack{\lambda \in \Lambda\setminus \corn(\Lambda)\\\lambda + e_i \in \corn(\Lambda)}}\!\!\! z_\lambda x^{\lambda+e_i} = \\ &=
\!\!\sum_{\substack{\lambda \in \Lambda\setminus \corn(\Lambda)\\\lambda + e_i \in \Lambda \setminus \corn(\Lambda)}}\!\!\! z_\lambda x^{\lambda+e_i} +
\Bigl(\!\sum_{\substack{\lambda \in \Lambda\setminus \corn(\Lambda)\\\lambda + e_i \in \corn(\Lambda)}}\!\!\! b_{\lambda + e_i} z_\lambda\Bigr) x^{\corn}.
\end{align*}
It follows that $x \in \Soc A_{\Lambda,B}$ if and only if
\begin{align}
z_\lambda &= 0 \quad \text{ for any } \lambda \in \Lambda \setminus (\corn(\Lambda) \cup \precorn(\Lambda)); \label{eqYDprecorn1} \\
\sum_{\substack{\lambda \in \precorn(\Lambda)\\\lambda + e_i \in \corn(\Lambda)}} b_{\lambda + e_i} z_\lambda &= 0 \quad \text{ for any } 1 \le i \le k. \label{eqYDprecorn2}
\end{align}
Since $\langle x^{\corn}\rangle \subseteq \Soc A_{\Lambda,B}$, the algebra $A_{\Lambda,B}$ is Gorenstein if and only if the system of linear equations~\eqref{eqYDprecorn1},~\eqref{eqYDprecorn2} has the only solution $z_\lambda = 0$, where $\lambda \in \Lambda \setminus \corn(\Lambda)$. Equivalently, the system of linear equations~\eqref{eqYDprecorn2} has the only solution $z_\lambda = 0$, where $\lambda \in \precorn(\Lambda)$.
\end{remark}

\begin{example}
In Example~\ref{YD1ex}, there are three precorner cells $(3,0)$, $(2,1)$ and $(0,2)$. System~\eqref{eqYDprecorn2} turns into $b_{(3,1)}z_{(2,1)} + b_{(1,2)}z_{(0,2)} = 0; \; b_{(3,1)}z_{(3,0)} = 0$; it has solutions $(z_{(3,0)}, z_{(2,1)}, z_{(0,2)}) = \alpha(0,1,-1)$, $\alpha \in \KK$. So, besides $xy^2 = x^3y$, $\Soc A$ also contains the element $(x^2y-y^2)$ and $A_{\Lambda, B}$ is not Gorenstein.
\end{example}

If $A_{\Lambda,B}$ is Gorenstein, or, equivalently, system~\eqref{eqYDprecorn2} has the only zero solution, then the number of equations is at least the number of variables. So, the necessary condition for $A_{\Lambda,B}$ to be Gorenstein is the inequality $|\precorn(\Lambda)| \le k$. It is a natural task to classify all diagrams $\Lambda$ of dimension $k$ such that $|\precorn(\Lambda)| \le k$. In the following example, we do this for $k = 2$.

\begin{example}
\label{YD2ex}
Given a 2-dimensional Young diagram $\Lambda$, we can order its corner cells from top left to bottom right. For any two consecutive corner cells, consider the hook joining them. It is easy to see that this hook contains at least one precorner cell. So, if $|\precorn(\Lambda)| \leqslant 2$, then $\Lambda$ cannot contain more than three corner cells. Consider the three possible cases separately; see Figure~\ref{YD2fig}.

1) Suppose that $\Lambda$ contains only one corner cell. Then, $\Lambda$ is a rectangle, which clearly has exactly two precorner cells (unless the rectangle is one-dimensional); see Figure~\ref{YD2fig}a).

2) Suppose that $\Lambda$ contains two corner cells. Note that the left adjacent cell of the top left corner (if it exists) should be a precorner cell. Similarly, the bottom adjacent cell of the bottom right corner (again, if it exists) should also be a precorner cell. So, the diagram $\Lambda$ cannot have both such cells (otherwise, $\Lambda$ has at least three precorner cells). Moreover, if both the height and the length of the hook joining the two corner cells are greater than two, then the hook has at least two precorner cells. Taking all this into account, we conclude that the Young diagrams with two corner cells and not more than two precorner cells are of the following types; see Figure~\ref{YD2fig}, b)--f). 

3) Suppose that $\Lambda$ contains three corner cells. In this case, the two precorner cells of $\Lambda$ should lie in the two hooks between the corner cells (one precorner cell in each hook). So, there are no cells left adjacent to the top left corner and bottom adjacent to the bottom right corner; also, the length or height of each hook is exactly 2. This implies that $\Lambda$ should be of one of the following types; see Figure~\ref{YD2fig}, g)--j).

We conclude that all $2$-dimensional Young diagrams with not more than $2$ corner cells are given in Figure~\ref{YD2fig}, a)-j).

\begin{figure}[h]
\begin{tikzpicture}[x=0.75pt,y=0.75pt,yscale=-1,xscale=1]
\def\d{10}
\coordinate (e1) at (\d,0);
\coordinate (e2) at (0,-\d);
\coordinate (O) at (0,0);
\draw ($(O)-2*(e1)$) node{a)};
\draw ($(O)$) + (-\d,\d) rectangle ++(2*5*\d-\d,-2*4*\d+\d);
\foreach \x/\y in {4/2, 3/3}
    \filldraw[precornfill] ($(O)+2*\x*(e1)+2*\y*(e2)$) + (-\d,-\d) rectangle ++(\d,\d);
\foreach \x/\y in {4/3}
    \filldraw[cornfill] ($(O)+2*\x*(e1)+2*\y*(e2)$) + (-\d,-\d) rectangle ++(\d,\d);
\coordinate (O) at ($(O)+14*(e1)$);
\draw ($(O)-2*(e1)$) node{b)};
\draw ($(O)$) + (-\d,\d) rectangle ++(2*5*\d-\d,-2*4*\d+\d);
\foreach \x/\y in {4/2, 3/3}
    \filldraw[precornfill] ($(O)+2*\x*(e1)+2*\y*(e2)$) + (-\d,-\d) rectangle ++(\d,\d);
\foreach \x/\y in {4/3, 0/4}
    \filldraw[cornfill] ($(O)+2*\x*(e1)+2*\y*(e2)$) + (-\d,-\d) rectangle ++(\d,\d);
\coordinate (O) at ($(O)+14*(e1)$);
\draw ($(O)-2*(e1)$) node{c)};
\draw ($(O)$) + (-\d,\d) rectangle ++(2*5*\d-\d,-2*4*\d+\d);
\foreach \x/\y in {4/2, 3/3}
    \filldraw[precornfill] ($(O)+2*\x*(e1)+2*\y*(e2)$) + (-\d,-\d) rectangle ++(\d,\d);
\foreach \x/\y in {4/3, 5/0}
    \filldraw[cornfill] ($(O)+2*\x*(e1)+2*\y*(e2)$) + (-\d,-\d) rectangle ++(\d,\d);
\coordinate (O) at ($(O)+16*(e1)$);
\draw ($(O)-2*(e1)$) node{d)};
\draw ($(O)+ (-\d,\d)$) -- ++($2*6*(e1)$) -- ++($2*(e2)$) -- ++($-2*5*(e1)$) -- ++($2*4*(e2)$) -- ++($-2*(e1)$) -- cycle;
\foreach \x/\y in {4/0, 0/3}
    \filldraw[precornfill] ($(O)+2*\x*(e1)+2*\y*(e2)$) + (-\d,-\d) rectangle ++(\d,\d);
\foreach \x/\y in {5/0, 0/4}
    \filldraw[cornfill] ($(O)+2*\x*(e1)+2*\y*(e2)$) + (-\d,-\d) rectangle ++(\d,\d);
\coordinate (O) at ($-16*(e2)$);
\draw ($(O)-2*(e1)$) node{e)};
\draw ($(O)+ (-\d,\d)$) -- ++($2*7*(e1)$) -- ++($2*(e2)$) -- ++($-2*3*(e1)$) -- ++($2*(e2)$) -- ++($-2*4*(e1)$) -- cycle;
\foreach \x/\y in {5/0, 2/1}
    \filldraw[precornfill] ($(O)+2*\x*(e1)+2*\y*(e2)$) + (-\d,-\d) rectangle ++(\d,\d);
\foreach \x/\y in {6/0, 3/1}
    \filldraw[cornfill] ($(O)+2*\x*(e1)+2*\y*(e2)$) + (-\d,-\d) rectangle ++(\d,\d);
\coordinate (O) at ($(O)+18*(e1)$);
\draw ($(O)-2*(e1)$) node{e')};
\draw ($(O)+ (-\d,\d)$) -- ++($2*7*(e1)$) -- ++($2*(e2)$) -- ++($-2*6*(e1)$) -- ++($2*(e2)$) -- ++($-2*1*(e1)$) -- cycle;
\foreach \x/\y in {5/0}
    \filldraw[precornfill] ($(O)+2*\x*(e1)+2*\y*(e2)$) + (-\d,-\d) rectangle ++(\d,\d);
\foreach \x/\y in {6/0, 0/1}
    \filldraw[cornfill] ($(O)+2*\x*(e1)+2*\y*(e2)$) + (-\d,-\d) rectangle ++(\d,\d);
\coordinate (O) at ($(O)+18*(e1)$);
\draw ($(O)-2*(e1)$) node{f)};
\draw ($(O)+ (-\d,\d)$) -- ++($2*7*(e2)$) -- ++($2*(e1)$) -- ++($-2*3*(e2)$) -- ++($2*(e1)$) -- ++($-2*4*(e2)$) -- cycle;
\foreach \x/\y in {0/5, 1/2}
    \filldraw[precornfill] ($(O)+2*\x*(e1)+2*\y*(e2)$) + (-\d,-\d) rectangle ++(\d,\d);
\foreach \x/\y in {0/6, 1/3}
    \filldraw[cornfill] ($(O)+2*\x*(e1)+2*\y*(e2)$) + (-\d,-\d) rectangle ++(\d,\d);
\coordinate (O) at ($(O)+8*(e1)$);
\draw ($(O)-2*(e1)$) node{f')};
\draw ($(O)+ (-\d,\d)$) -- ++($2*7*(e2)$) -- ++($2*(e1)$) -- ++($-2*6*(e2)$) -- ++($2*(e1)$) -- ++($-2*1*(e2)$) -- cycle;
\foreach \x/\y in {0/5}
    \filldraw[precornfill] ($(O)+2*\x*(e1)+2*\y*(e2)$) + (-\d,-\d) rectangle ++(\d,\d);
\foreach \x/\y in {0/6, 1/0}
    \filldraw[cornfill] ($(O)+2*\x*(e1)+2*\y*(e2)$) + (-\d,-\d) rectangle ++(\d,\d);
\coordinate (O) at ($-32*(e2)$);
\draw ($(O)-2*(e1)$) node{g)};
\draw ($(O)+ (-\d,\d)$) -- ++($2*6*(e1)$) -- ++($2*(e2)$) -- ++($-2*5*(e1)$) -- ++($2*4*(e2)$) -- ++($-2*(e1)$) -- cycle;
\foreach \x/\y in {4/0, 0/3}
    \filldraw[precornfill] ($(O)+2*\x*(e1)+2*\y*(e2)$) + (-\d,-\d) rectangle ++(\d,\d);
\foreach \x/\y in {5/0, 0/4, 1/1}
    \filldraw[cornfill] ($(O)+2*\x*(e1)+2*\y*(e2)$) + (-\d,-\d) rectangle ++(\d,\d);
\coordinate (O) at ($(O)+16*(e1)$);
\draw ($(O)-2*(e1)$) node{h)};
\draw ($(O)$) + (-\d,\d) rectangle ++(2*5*\d-\d,-2*4*\d+\d);
\foreach \x/\y in {4/2, 3/3}
    \filldraw[precornfill] ($(O)+2*\x*(e1)+2*\y*(e2)$) + (-\d,-\d) rectangle ++(\d,\d);
\foreach \x/\y in {4/3, 5/0, 0/4}
    \filldraw[cornfill] ($(O)+2*\x*(e1)+2*\y*(e2)$) + (-\d,-\d) rectangle ++(\d,\d);
\coordinate (O) at ($(O)+16*(e1)$);
\draw ($(O)-2*(e1)$) node{i)};
\draw ($(O)+ (-\d,\d)$) -- ++($2*7*(e1)$) -- ++($2*(e2)$) -- ++($-2*3*(e1)$) -- ++($2*(e2)$) -- ++($-2*4*(e1)$) -- cycle;
\foreach \x/\y in {5/0, 2/1}
    \filldraw[precornfill] ($(O)+2*\x*(e1)+2*\y*(e2)$) + (-\d,-\d) rectangle ++(\d,\d);
\foreach \x/\y in {6/0, 3/1, 0/2}
    \filldraw[cornfill] ($(O)+2*\x*(e1)+2*\y*(e2)$) + (-\d,-\d) rectangle ++(\d,\d);
\coordinate (O) at ($(O)+18*(e1)$);
\draw ($(O)-2*(e1)$) node{j)};
\draw ($(O)+ (-\d,\d)$) -- ++($2*7*(e2)$) -- ++($2*(e1)$) -- ++($-2*3*(e2)$) -- ++($2*(e1)$) -- ++($-2*4*(e2)$) -- cycle;
\foreach \x/\y in {0/5, 1/2}
    \filldraw[precornfill] ($(O)+2*\x*(e1)+2*\y*(e2)$) + (-\d,-\d) rectangle ++(\d,\d);
\foreach \x/\y in {0/6, 1/3, 2/0}
    \filldraw[cornfill] ($(O)+2*\x*(e1)+2*\y*(e2)$) + (-\d,-\d) rectangle ++(\d,\d);

\end{tikzpicture}
\caption{Young diagrams in Example~\ref{YD2ex}.}
\label{YD2fig}
\end{figure}
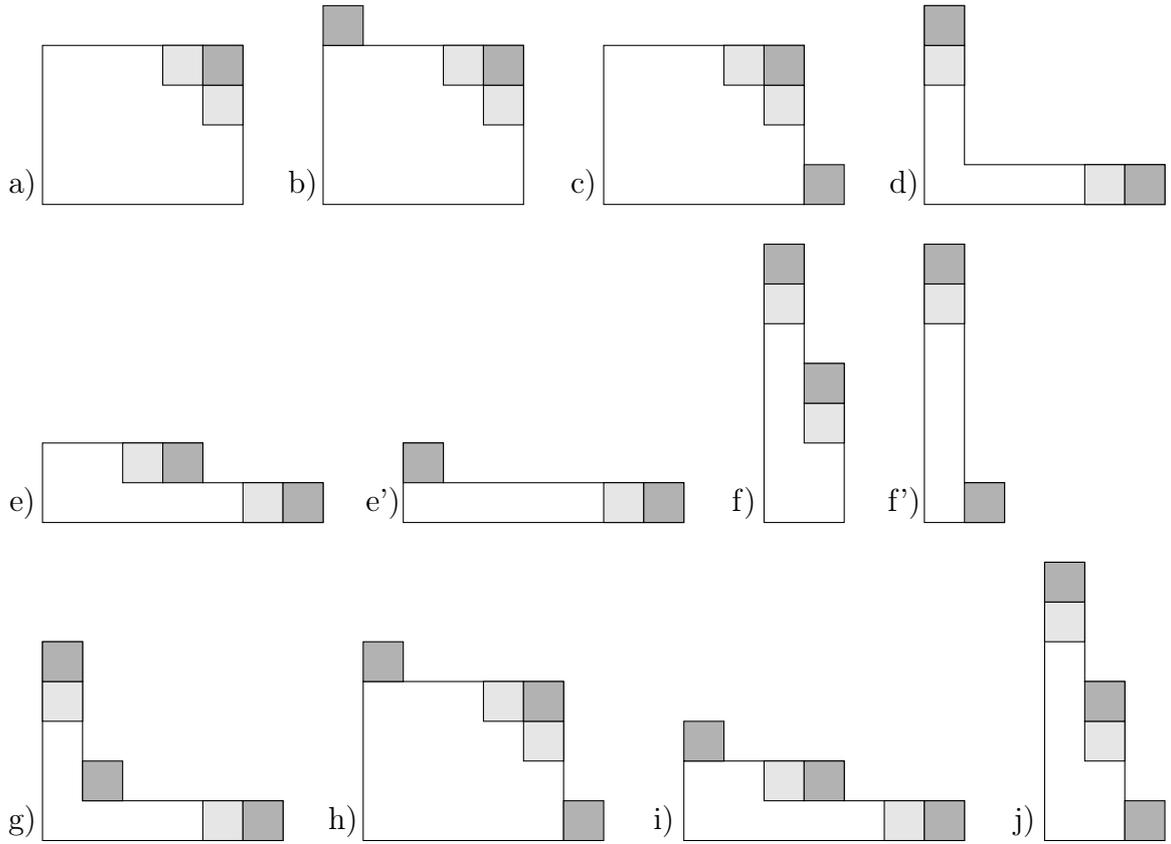
\end{example}

Example~\ref{YD2ex} shows that $|\precorn(\Lambda)| \ge 2$ for any non-exceptional two-dimensional Young diagram (the only Young diagrams with one precorner cell shown in Figure~\ref{YD2fig},~e'),~f') are exceptional). 

\begin{proposition}
\label{precorn_prop}
If $\Lambda \subseteq \ZZ^k$ is a non-exceptional $k$-dimensional Young diagram, then $|\precorn(\Lambda)| \ge k$.
\end{proposition}

\begin{proof}

Let us reformulate the notion of exceptional or non-exceptional diagrams in combinatorial terms. Consider an undirected graph with loops $G_\Lambda$ with the vertex set $V = \{1, \ldots, k\}$ and the edge set $E = \{(i,j) \mid e_i+e_j \in \Lambda\} \subseteq V \times V$. By definition, the diagram $\Lambda$ is exceptional with respect to the $i$-th coordinate if and only if the vertex~$i$ has degree~$0$. Indeed, $e_i$ is the unique cell with $\lambda_i \ne 0$ if and only if there is no $e_i+e_j \in \Lambda$ (including the case $j = i$).

Let us prove the proposition by induction on $k$.
%
As above, we denote coordinates in $\ZZ^k$ by $\lambda_j$, $1 \le j \le k$.

If the graph $G_\Lambda$ is disconnected, then $\Lambda$ is a union of two Young diagrams $\Lambda_1, \Lambda_2$ of dimensions $k_1, k_2$ with $k_1+k_2 = k$ and $\Lambda_1 \cap \Lambda_2 = 0$. Then \[|\precorn(\Lambda)| = |\precorn(\Lambda_1) \cup \precorn(\Lambda_2)| \ge k_1+k_2 = k\] by induction hypothesis. So, we can assume that $G_\Lambda$ is connected.

Notice that any layer $\Lambda_{j,c} = \Lambda \cap \{\lambda_j = c\}$, $c \in \Zgezero$, is a $(k-1)$-dimensional Young diagram. First let us show that there exists $1 \le i \le k$ such that the bottom layer $\Lambda_{i,0}$ is non-exceptional. Indeed, the graph of $\Lambda_{i,0}$ is obtained from $G_\Lambda$ by removing the vertex~$i$. Since $G_\Lambda$ is connected, one can remove a vertex in such a way that the result remains a connected graph; in particular, any vertex remains of degree at least one. The exception for this argument is the graph with two vertices, one edge and no loops, i.e. the two-dimensional Young diagram with 4 cells in a square, which has two precorner cells as required.

For any precorner cell of a <<horizontal>> layer $\lambda \in \Lambda_{i,c}$, the <<vertical>> column $(\lambda + \ZZ e_i) \cap \Lambda$ contains at least one precorner cell of $\Lambda$. Indeed, either the highest cell in this column or the cell under it is a precorner cell of $\Lambda$. Since $\Lambda_{i,0}$ is non-exceptional, there exists at least $k-1$ precorner cells of~$\Lambda_{i,0}$ by induction hypothesis, and thus there is at least $k-1$ precorner cells of~$\Lambda$ in their vertical columns. It remains to prove that there is at least one more.

If there is a corner cell of the <<bottom>> layer $\mu \in \Lambda_{i,0}$ that is not unique in its column, then this column contains at least one precorner cell of~$\Lambda$. Namely, it is the cell under the highest cell of this column. So let any corner cell of the <<bottom>> layer be unique in its column. Then precorner cells of the first <<horizontal>> layer $\Lambda_{i,1}$ and precorner cells of the bottom <<horizontal>> layer $\Lambda_{i,0}$ are in different <<vertical>> columns. If there is at least one precorner cell in~$\Lambda_{i,1}$, then it follows that there is at least one more precorner cell of~$\Lambda$ in its column. Otherwise, $\Lambda_{i,1}$ has a unique cell, whence $\Lambda$ is exceptional with respect to the $i$-th coordinate; a contradiction.
%
%
%
\end{proof}

\begin{example}
All local algebras of dimension at most~6 are listed in~\cite[Table~1]{AZa}. The Gorenstein algebras among them are the algebras No.~1--3, 5-6, 9-10, 14, 18-21, 30, 38. These algebras are listed in Table 1 below. It turns out that they have the form $A_{\Lambda,B}$; see Figure~\ref{YDtablefig}. Here No.~38 is a Young diagram in $\Zgezero^4$ consisting of $0, e_i, 2e_i$, $1 \le i \le 4$. In all the cases, system~\eqref{eqYDprecorn2} consists of equations of the form $z_\lambda = 0$, $\lambda \in \precorn(\Lambda)$, and has the only zero solution.

\begin{center}
\begin{tabular}{ |c|c|c| }
 \hline
 № & Algebra & Dimension\\
 \hline\hline
 1& $\KK$ & $1$\\
 \hline
 2& $\KK[x_1]/(x_1^2)$ & 2\\
 \hline
 3& $\KK[x_1]/(x_1^3)$ & 3\\
 \hline
 5& $\KK[x_1]/(x_1^4)$ & 4\\
 \hline
 6& $\KK[x_1,x_2]/(x_1x_2, x_1^2-x_2^2)$ & 4\\
 \hline
 9& $\KK[x_1]/(x_1^5)$ & 5\\
 \hline
 10& $\KK[x_1,x_2]/(x_1x_2,x_1^3-x_2^2)$ & 5\\
 \hline
 14& $\KK[x_1,x_2,x_3]/(x_1x_2,x_1x_3,x_2x_3,x_1^2-x_2^2,x_1^2-x_3^2)$ & 5\\
 \hline
 18 & $\KK[x_1]/(x_1^6)$ & 6\\
 \hline
 19 & $\KK[x_1,x_2]/(x_1x_2,x_1^4-x_2^2)$ & 6\\
 \hline
 20 & $\KK[x_1,x_2]/(x_1x_2,x_1^3-x_2^3)$ & 6\\
 \hline
 21& $\KK[x_1,x_2]/(x_1^3,x_2^2)$ & 6\\
 \hline
 30& $\KK[x_1,x_2,x_3]/(x_1^2,x_2^2,x_1x_3,x_2x_3,x_1x_2-x_3^3)$ & 6\\
 \hline
 38& $\KK[x_1,x_2,x_3,x_4]/(x_i^2-x_j^2, x_ix_j, i\ne j)$ & 6\\
 \hline
\end{tabular}

\vspace{0.3cm}
Table 1: Local Gorenstein algebras of dimensions up to $6$
\end{center}

\begin{figure}[h]
\begin{tikzpicture}[x=0.75pt,y=0.75pt,yscale=-1,xscale=1]
\def\d{10}
\coordinate (e1) at (\d,0);
\coordinate (e2) at (0,-\d);
\coordinate (e3) at (-\d*0.4,\d*0.6);
\coordinate (O) at (0,0);
\draw ($(O)-2*(e2)$) node{No.~1};
\filldraw[cornfill] ($(O)+2*0*(e1)+2*0*(e2)$) + (-\d,-\d) rectangle ++(\d,\d);
\coordinate (O) at ($(O)+6*(e1)$);
\draw ($(O)-2*(e2)$) node{No.~2};
\filldraw[precornfill] ($(O)+2*0*(e1)+2*0*(e2)$) + (-\d,-\d) rectangle ++(\d,\d);
\filldraw[cornfill] ($(O)+2*1*(e1)+2*0*(e2)$) + (-\d,-\d) rectangle ++(\d,\d);
\coordinate (O) at ($(O)+8*(e1)$);
\draw ($(O)-2*(e2)$) node{No.~3};
\draw ($(O)+2*0*(e1)+2*0*(e2)$) + (-\d,-\d) rectangle ++(\d,\d);
\filldraw[precornfill] ($(O)+2*1*(e1)+2*0*(e2)$) + (-\d,-\d) rectangle ++(\d,\d);
\filldraw[cornfill] ($(O)+2*2*(e1)+2*0*(e2)$) + (-\d,-\d) rectangle ++(\d,\d);
\coordinate (O) at ($(O)+10*(e1)$);
\draw ($(O)-2*(e2)$) node{No.~5};
\foreach \x in {0,1}
    \draw ($(O)+2*\x*(e1)+2*0*(e2)$) + (-\d,-\d) rectangle ++(\d,\d);
\filldraw[precornfill] ($(O)+2*2*(e1)+2*0*(e2)$) + (-\d,-\d) rectangle ++(\d,\d);
\filldraw[cornfill] ($(O)+2*3*(e1)+2*0*(e2)$) + (-\d,-\d) rectangle ++(\d,\d);
\coordinate (O) at ($(O)+12*(e1)$);
\draw ($(O)-2*(e2)$) node{No.~6};
\draw ($(O)+2*0*(e1)+2*0*(e2)$) + (-\d,-\d) rectangle ++(\d,\d);
\foreach \x/\y in {0/1, 1/0}
    \filldraw[precornfill] ($(O)+2*\x*(e1)+2*\y*(e2)$) + (-\d,-\d) rectangle ++(\d,\d);
\foreach \x/\y in {0/2, 2/0}
    \filldraw[cornfill] ($(O)+2*\x*(e1)+2*\y*(e2)$) + (-\d,-\d) rectangle ++(\d,\d);
\coordinate (O) at ($(O)+10*(e1)$);
\draw ($(O)-2*(e2)$) node{No.~9};
\foreach \x in {0,1,2}
    \draw ($(O)+2*\x*(e1)+2*0*(e2)$) + (-\d,-\d) rectangle ++(\d,\d);
\filldraw[precornfill] ($(O)+2*2*(e1)+2*0*(e2)$) + (-\d,-\d) rectangle ++(\d,\d);
\filldraw[cornfill] ($(O)+2*3*(e1)+2*0*(e2)$) + (-\d,-\d) rectangle ++(\d,\d);
\coordinate (O) at ($-12*(e2)$);
\draw ($(O)-2*(e2)$) node{No.~10};
\foreach \x in {0,1}
    \draw ($(O)+2*\x*(e1)+2*0*(e2)$) + (-\d,-\d) rectangle ++(\d,\d);
\foreach \x/\y in {0/1, 2/0}
    \filldraw[precornfill] ($(O)+2*\x*(e1)+2*\y*(e2)$) + (-\d,-\d) rectangle ++(\d,\d);
\foreach \x/\y in {0/2, 3/0}
    \filldraw[cornfill] ($(O)+2*\x*(e1)+2*\y*(e2)$) + (-\d,-\d) rectangle ++(\d,\d);
\coordinate (O) at ($(O)+16*(e1)+2*(e2)$);
\draw ($(O)-4*(e2)+2*(e1)$) node{No.~14};
\foreach \x/\y/\z in {0/-1/0,-1/0/0}
    \filldraw[precornfill] ($(O)+2*\x*(e1)+2*\y*(e2)+2*\z*(e3)$) + (-\d,-\d) rectangle ++(\d,\d);
\foreach \x/\y/\z in {0/0/0,-1/0/1}
    \filldraw[fill=black!6, draw = black] ($(O)+(-\d,\d)+2*\x*(e1)+2*\y*(e2)+2*\z*(e3)$) -- ++($2*(e1)$) -- ++($-2*(e3)$) -- ++($-2*(e1)$) -- cycle;
\foreach \x/\y/\z in {0/0/0,0/-1/1}
    \filldraw[fill=black!17, draw = black] ($(O)+(-\d,\d)+2*\x*(e1)+2*\y*(e2)+2*\z*(e3)$) -- ++($2*(e2)$) -- ++($-2*(e3)$) -- ++($-2*(e2)$) -- cycle;
\foreach \x/\y/\z in {1/-1/0,-1/1/0,-1/-1/2}
    \filldraw[cornfill] ($(O)+2*\x*(e1)+2*\y*(e2)+2*\z*(e3)$) + (-\d,-\d) rectangle ++(\d,\d);
\foreach \x/\y/\z in {1/0/0,-1/0/2,-1/2/0}
    \filldraw[fill=black!25, draw = black] ($(O)+(-\d,\d)+2*\x*(e1)+2*\y*(e2)+2*\z*(e3)$) -- ++($2*(e1)$) -- ++($-2*(e3)$) -- ++($-2*(e1)$) -- cycle;
\foreach \x/\y/\z in {0/1/0,0/-1/2,2/-1/0}
    \filldraw[fill=black!40, draw = black] ($(O)+(-\d,\d)+2*\x*(e1)+2*\y*(e2)+2*\z*(e3)$) -- ++($2*(e2)$) -- ++($-2*(e3)$) -- ++($-2*(e2)$) -- cycle;
\coordinate (O) at ($(O)+10*(e1)-2*(e2)$);
\draw ($(O)-2*(e2)$) node{No.~18};
\foreach \x in {0,1,2,3}
    \draw ($(O)+2*\x*(e1)+2*0*(e2)$) + (-\d,-\d) rectangle ++(\d,\d);
\filldraw[precornfill] ($(O)+2*3*(e1)+2*0*(e2)$) + (-\d,-\d) rectangle ++(\d,\d);
\filldraw[cornfill] ($(O)+2*4*(e1)+2*0*(e2)$) + (-\d,-\d) rectangle ++(\d,\d);
\coordinate (O) at ($(O)+14*(e1)$);
\draw ($(O)-2*(e2)$) node{No.~19};
\foreach \x in {0,1,2}
    \draw ($(O)+2*\x*(e1)+2*0*(e2)$) + (-\d,-\d) rectangle ++(\d,\d);
\foreach \x/\y in {0/1, 3/0}
    \filldraw[precornfill] ($(O)+2*\x*(e1)+2*\y*(e2)$) + (-\d,-\d) rectangle ++(\d,\d);
\foreach \x/\y in {0/2, 4/0}
    \filldraw[cornfill] ($(O)+2*\x*(e1)+2*\y*(e2)$) + (-\d,-\d) rectangle ++(\d,\d);
\coordinate (O) at ($-28*(e2)$);
\draw ($(O)-2*(e2)$) node{No.~20};
\foreach \x/\y in {0/0, 0/1, 1/0}
    \draw ($(O)+2*\x*(e1)+2*\y*(e2)$) + (-\d,-\d) rectangle ++(\d,\d);
\foreach \x/\y in {0/2, 2/0}
    \filldraw[precornfill] ($(O)+2*\x*(e1)+2*\y*(e2)$) + (-\d,-\d) rectangle ++(\d,\d);
\foreach \x/\y in {0/3, 3/0}
    \filldraw[cornfill] ($(O)+2*\x*(e1)+2*\y*(e2)$) + (-\d,-\d) rectangle ++(\d,\d);
\coordinate (O) at ($(O)+12*(e1)$);
\draw ($(O)-2*(e2)$) node{No.~21};
\foreach \x/\y in {0/0, 0/1, 1/0}
    \draw ($(O)+2*\x*(e1)+2*\y*(e2)$) + (-\d,-\d) rectangle ++(\d,\d);
\foreach \x/\y in {1/1, 2/0}
    \filldraw[precornfill] ($(O)+2*\x*(e1)+2*\y*(e2)$) + (-\d,-\d) rectangle ++(\d,\d);
\foreach \x/\y in {2/1}
    \filldraw[cornfill] ($(O)+2*\x*(e1)+2*\y*(e2)$) + (-\d,-\d) rectangle ++(\d,\d);
\coordinate (O) at ($(O)+14*(e1)+3.5*(e2)$);
\draw ($(O)-5.5*(e2)$) node{No.~30};
\draw ($(O)+2*-1*(e1)+2*0*(e2)+2*0*(e3)$) + (-\d,-\d) rectangle ++(\d,\d);
\draw ($(O)+(-\d,\d)$) -- ++($2*(e2)$) -- ++($-2*(e3)$) -- ++($-2*(e2)$) -- cycle;
\foreach \x/\y/\z in {-1/1/0,-1/-1/1}
    \filldraw[precornfill] ($(O)+2*\x*(e1)+2*\y*(e2)+2*\z*(e3)$) + (-\d,-\d) rectangle ++(\d,\d);
\foreach \x/\y/\z in {0/0/0,-1/0/1}
    \filldraw[fill=black!6, draw = black] ($(O)+(-\d,\d)+2*\x*(e1)+2*\y*(e2)+2*\z*(e3)$) -- ++($2*(e1)$) -- ++($-2*(e3)$) -- ++($-2*(e1)$) -- cycle;
\foreach \x/\y/\z in {0/1/0,1/-1/0}
    \filldraw[fill=black!17, draw = black] ($(O)+(-\d,\d)+2*\x*(e1)+2*\y*(e2)+2*\z*(e3)$) -- ++($2*(e2)$) -- ++($-2*(e3)$) -- ++($-2*(e2)$) -- cycle;
\foreach \x/\y/\z in {0/-1/1,-1/2/0}
    \filldraw[cornfill] ($(O)+2*\x*(e1)+2*\y*(e2)+2*\z*(e3)$) + (-\d,-\d) rectangle ++(\d,\d);
\foreach \x/\y/\z in {0/0/1,-1/3/0}
    \filldraw[fill=black!25, draw = black] ($(O)+(-\d,\d)+2*\x*(e1)+2*\y*(e2)+2*\z*(e3)$) -- ++($2*(e1)$) -- ++($-2*(e3)$) -- ++($-2*(e1)$) -- cycle;
\foreach \x/\y/\z in {0/2/0,1/-1/1}
    \filldraw[fill=black!40, draw = black] ($(O)+(-\d,\d)+2*\x*(e1)+2*\y*(e2)+2*\z*(e3)$) -- ++($2*(e2)$) -- ++($-2*(e3)$) -- ++($-2*(e2)$) -- cycle;
\coordinate (O) at ($(O)+14*(e1)$);
\draw ($(O)-5.5*(e2)$) node{No.~38};
\coordinate (e3) at (-\d*0.5,\d*0.5);
\coordinate (e4) at (\d*0.5,\d*0.5);
\filldraw[fill=black!6, draw = black] ($(O)+2*0*(e1)+2*0*(e2)+2*0*(e3)+2*0*(e4)$) -- ++($-2*(e1)$) -- ++($2*(e3)$) -- ++($2*(e1)$) -- cycle;
\filldraw[fill=black!6, draw = black] ($(O)+2*0*(e1)+2*0*(e2)+2*0*(e3)+2*0*(e4)$) -- ++($2*(e1)$) -- ++($2*(e4)$) -- ++($-2*(e1)$) -- cycle;
\filldraw[fill=black!18, draw = black] ($(O)+2*0*(e1)+2*0*(e2)+2*0*(e3)+2*0*(e4)$) -- ++($2*(e1)$) -- ++($-2*(e3)$) -- ++($-2*(e1)$) -- cycle;
\filldraw[fill=black!18, draw = black] ($(O)+2*0*(e1)+2*0*(e2)+2*0*(e3)+2*0*(e4)$) -- ++($-2*(e1)$) -- ++($-2*(e4)$) -- ++($2*(e1)$) -- cycle;
\filldraw[fill=black!25, draw = black] ($(O)+2*0*(e1)+2*0*(e2)+2*1*(e3)+2*0*(e4)$) -- ++($-2*(e1)$) -- ++($2*(e3)$) -- ++($2*(e1)$) -- cycle;
\filldraw[fill=black!25, draw = black] ($(O)+2*0*(e1)+2*0*(e2)+2*0*(e3)+2*1*(e4)$) -- ++($2*(e1)$) -- ++($2*(e4)$) -- ++($-2*(e1)$) -- cycle;
\filldraw[fill=black!40, draw = black] ($(O)+2*0*(e1)+2*0*(e2)+2*-1*(e3)+2*0*(e4)$) -- ++($2*(e1)$) -- ++($-2*(e3)$) -- ++($-2*(e1)$) -- cycle;
\filldraw[fill=black!40, draw = black] ($(O)+2*0*(e1)+2*0*(e2)+2*0*(e3)+2*-1*(e4)$) -- ++($-2*(e1)$) -- ++($-2*(e4)$) -- ++($2*(e1)$) -- cycle;
\filldraw[fill=black!15, draw = black] ($(O)+2*0*(e1)+2*0*(e2)+2*0*(e3)+2*0*(e4)$) -- ++($-2*(e2)$) -- ++($2*(e3)$) -- ++($2*(e2)$) -- cycle;
\filldraw[fill=black!11, draw = black] ($(O)+2*0*(e1)+2*0*(e2)+2*0*(e3)+2*0*(e4)$) -- ++($-2*(e2)$) -- ++($2*(e4)$) -- ++($2*(e2)$) -- cycle;
\filldraw[fill=black!11, draw = black] ($(O)+2*0*(e1)+2*0*(e2)+2*0*(e3)+2*0*(e4)$) -- ++($2*(e2)$) -- ++($-2*(e3)$) -- ++($-2*(e2)$) -- cycle;
\filldraw[fill=black!15, draw = black] ($(O)+2*0*(e1)+2*0*(e2)+2*0*(e3)+2*0*(e4)$) -- ++($2*(e2)$) -- ++($-2*(e4)$) -- ++($-2*(e2)$) -- cycle;
\filldraw[fill=black!45, draw = black] ($(O)+2*0*(e1)+2*0*(e2)+2*1*(e3)+2*0*(e4)$) -- ++($-2*(e2)$) -- ++($2*(e3)$) -- ++($2*(e2)$) -- cycle;
\filldraw[fill=black!27, draw = black] ($(O)+2*0*(e1)+2*0*(e2)+2*0*(e3)+2*1*(e4)$) -- ++($-2*(e2)$) -- ++($2*(e4)$) -- ++($2*(e2)$) -- cycle;
\filldraw[fill=black!27, draw = black] ($(O)+2*0*(e1)+2*0*(e2)+2*-1*(e3)+2*0*(e4)$) -- ++($2*(e2)$) -- ++($-2*(e3)$) -- ++($-2*(e2)$) -- cycle;
\filldraw[fill=black!45, draw = black] ($(O)+2*0*(e1)+2*0*(e2)+2*0*(e3)+2*-1*(e4)$) -- ++($2*(e2)$) -- ++($-2*(e4)$) -- ++($-2*(e2)$) -- cycle;
\filldraw[fill=black!30, draw = black] ($(O)+2*0*(e1)+2*0*(e2)+2*2*(e3)+2*0*(e4)$) -- ++($-2*(e1)$) -- ++($-2*(e2)$) -- ++($2*(e1)$) -- cycle;
\filldraw[fill=black!30, draw = black] ($(O)+2*0*(e1)+2*0*(e2)+2*0*(e3)+2*2*(e4)$) -- ++($2*(e1)$) -- ++($-2*(e2)$) -- ++($-2*(e1)$) -- cycle;
\filldraw[fill=black!30, draw = black] ($(O)+2*0*(e1)+2*0*(e2)+2*-2*(e3)+2*0*(e4)$) -- ++($2*(e1)$) -- ++($2*(e2)$) -- ++($-2*(e1)$) -- cycle;
\filldraw[fill=black!30, draw = black] ($(O)+2*0*(e1)+2*0*(e2)+2*0*(e3)+2*-2*(e4)$) -- ++($-2*(e1)$) -- ++($2*(e2)$) -- ++($2*(e1)$) -- cycle;
\end{tikzpicture}
\caption{Young diagrams for Gorenstein local algebras of dimension $\le 6$.}
\label{YDtablefig}
\end{figure}
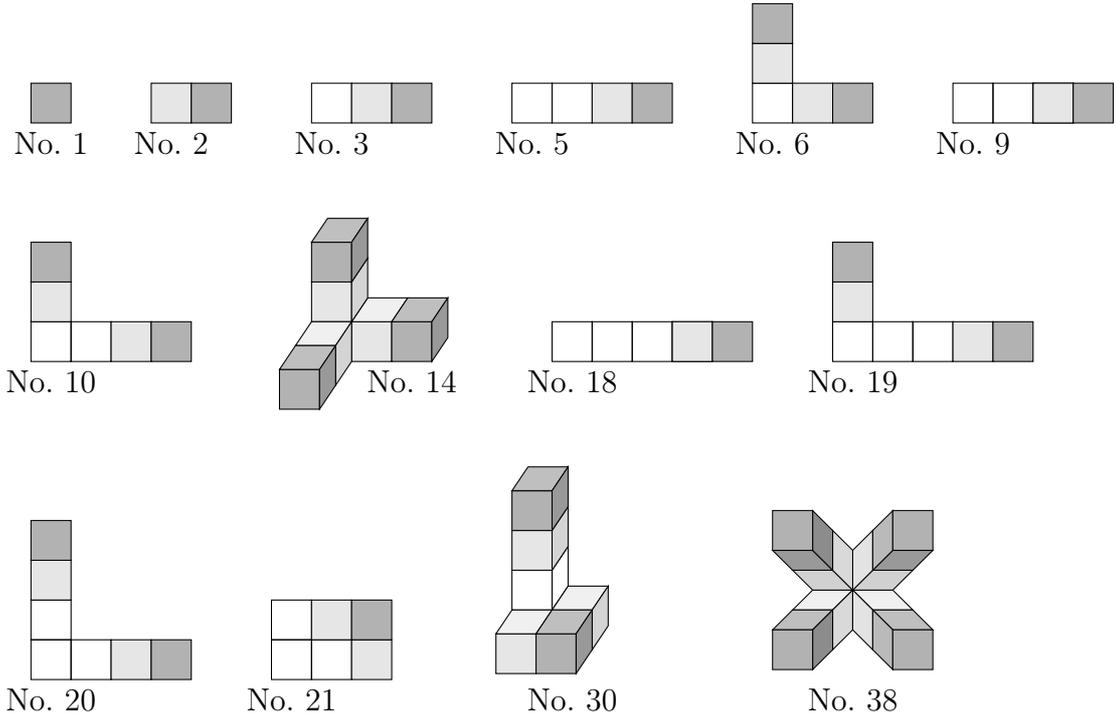
\end{example}

\begin{proposition}
\label{propYD}
Consider an $H$-pair $(A_{\Lambda,B}, U_{\Lambda,B})$ and the corresponding projective hypersurface $X = \{f(z) = 0\} \subseteq \PP^n = \PP(A_{\Lambda,B})$ of degree $d$ admitting an additive action. Then
\begin{equation}
\label{eqpropYD_d}
d = \max_{\substack{\mu \in \corn(\Lambda)\\b_\mu \ne 0}} |\lambda|,
\end{equation}
and polynomials $f_d$, $f_{d-1}$ in decomposition~\eqref{eqsumf} are as follows:
\begin{equation}
\label{eqpropYD}
\begin{aligned}
f_d &= \frac{(-1)^{d-1}}{d}\sum_{\substack{\mu \in \corn(\Lambda)\\ |\mu| = d}} c_\mu b_\mu z^\mu,\\
f_{d-1} &= \frac{(-1)^{d-2}}{d-1}\sum_{\substack{\mu \in \corn(\Lambda)\\ |\mu| = d-1}} c_\mu b_\mu z^\mu +
        (-1)^{d-2}\sum_{\substack{\mu \in \corn(\Lambda)\\ |\mu| = d}}\sum_{\substack{|\nu|=2\\ \nu \lecur \mu}} c_{\mu-\nu}b_\mu z_\nu z^{\mu-\nu},
\end{aligned}
\end{equation}
where $c_\mu = \frac{(\mu_1 + \ldots + \mu_k)!}{\mu_1! \ldots \mu_k!}$ denotes the multinomial coefficient.
\end{proposition}

\begin{remark}
According to Proposition~\ref{prop_reduct}, the non-degenerate hypersurface $Z$ corresponding to the reduced $H$-pair $(\widetilde{A_{\Lambda,B}}, \widetilde{U_{\Lambda,B}})$ associated with $(\Lambda, B)$ has the same equation and degree as~$X$.
\end{remark}

\begin{proof}
According to~\cite[Theorem~2.14]{AZa}, the equation of $X$ is
\[z_0^d\pi\left(\ln\Bigl(1+\frac{z}{z_0}\Bigr)\right) = 0,\]
where $z_0 + z \in \KK \oplus \mm = A_{\Lambda,B}$ and $\pi\colon \mm \to \mm/U_{\Lambda,B} \cong \KK$ is the projection. For $z \in \mm$, let $\pi(z)$ equal the coordinate $z_{\corn}$ in decomposition~\eqref{eqYDcoord}. Notice that
\[z_0^d\ln\Bigl(1+\frac{z}{z_0}\Bigr) = \sum_{s=1}^d \frac{(-1)^{s-1}}{s} z_0^{d-s} z^s.\]

For any $s \ge 2$, the expanding of $z^s$ gives the sum of $z_{\lambda^{(1)}}\ldots z_{\lambda^{(s)}} x^{\lambda^{(1)}+\ldots+\lambda^{(s)}}$ over all tuples $(\lambda^{(1)}, \ldots, \lambda^{(s)}) \in (\Lambda \setminus \{0\})^s$; see~\eqref{eqYDcoord}.
The application of~$\pi$ leaves only summands with $\lambda^{(1)}+\ldots+\lambda^{(s)} \in \corn(\Lambda)$. It follows that
\[f_s = \sum_{\mu \in \corn(\Lambda)} b_\mu \!\!\sum_{\lambda^{(1)}+\ldots+\lambda^{(s)} = \mu}\!\! z_{\lambda^{(1)}}\ldots z_{\lambda^{(s)}}.\]
This proves formula~\eqref{eqpropYD_d}. For the polynomial $f_d$ we have $|\mu| = d$ and $|\lambda^{(i)}| = 1$ for any $1 \le i \le d$, i.e., in every summand the coordinates $z_{\lambda^{(i)}}$ are $z_1, \ldots, z_k$ with the product $z^\mu = z_1^{\mu_1}\ldots z_k^{\mu_k}$. The number of appropriate tuples equals $c_\mu$, so we obtain the formula for $f_d$ in~\eqref{eqpropYD}.

For $s = d-1$, we have two cases. If $|\mu| = d-1$, then $|\lambda^{(i)}| = 1$ for all $1 \le i \le d-1$, and we obtain the first sum in the formula for $f_{d-1}$ in~\eqref{eqpropYD}. If $|\mu| = d$ then $|\lambda^{(i)}| = 1$ for all but one $1 \le i \le d-1$ and there is one $\lambda^{(i)} = \nu$ with $|\nu| = 2$. The number of appropriate tuples equals $(d-1)c_{\mu-\nu}$, which gives the second sum in the formula for $f_{d-1}$ in~\eqref{eqpropYD}.
\end{proof}

Let us give an application of Proposition~\ref{propYD}.

\begin{proposition}
\label{prYD}
Let $X$ be the projective hypersurface admitting an additive action associated with an $H$-pair $(A_{\Lambda, B}, U_{\Lambda, B})$. Suppose there exists $1 \le i \le k$ such that
\smallskip
\begin{itemize}
    \item $\mu_i \ge 3$ for any $\mu \in \corn(\Lambda)$ with $|\mu| = d$;
    \item $\mu_i \ge 1$ for any $\mu \in \corn(\Lambda)$ with $|\mu| = d - 1$.
\end{itemize}
\smallskip
Then $X$ is non-normal.
\end{proposition}

\begin{remark}
According to Lemma~\ref{l1}, the non-degenerate hypersurface $Z$ corresponding to the reduced $H$-pair $(\widetilde{A_{\Lambda,B}}, \widetilde{U_{\Lambda,B}})$ associated with $(\Lambda, B)$ is normal if and only if $X$ is normal.
\end{remark}

\begin{proof}
Under these conditions, $f_d$ is divisible by $z_i^3$, and $f_{d-1}$ is divisible by~$z_i$ according to Proposition~\ref{propYD}. So $\widetilde f_d$ and $f_{d-1}$ are not coprime, and $X$ is non-normal by Proposition~\ref{propnorm}.
\end{proof}

Now we come to a series of examples illustrating the above constructions and results.

\begin{example}
\label{exYDrays}
Let $\Lambda$ consist of ``rays'' of lengths $d_1 \ge \ldots \ge d_k \ge 2$, i.e., 
the diagram~$\Lambda$ is given by the set of corner cells $d_ie_i$, $1 \le i \le k$; see Figure~\ref{YDraysfig}.
Consider any set of non-zero constants $B = \{b_1, \ldots, b_k\}$, the associated $H$-pair $(A_{\Lambda,B}, U_{\Lambda,B})$, and the projective hypersurface $X$ admitting an additive action. Denote $d = d_1 = \max d_i$. Notice that $(A_{\Lambda,B}, U_{\Lambda,B}) = (\widetilde{A_{\Lambda,B}}, \widetilde{U_{\Lambda,B}})$ since the system of linear equations~\eqref{eqYDprecorn2} is $b_iz_{(d_i - 1)e_i} = 0$, $1 \le i \le k$.
\begin{figure}[h]
\begin{tikzpicture}[x=0.75pt,y=0.75pt,yscale=-1,xscale=1]
\def\d{10}
\coordinate (e1) at (\d,0);
\coordinate (e2) at (0,-\d);
\coordinate (e3) at (-\d*0.5,\d*0.6);
\coordinate (O) at (0,0);
\draw ($(O)+(-\d,\d)-2*(e1)$) -- ++($2*6*(e1)$) -- ++($-2*(e2)$) -- ++($-2*5*(e1)$) -- ++($2*5*(e2)$) -- ++($-2*(e1)$) -- cycle;
\draw ($(O)+(-\d,\d)-2*(e2)$) -- ++($2*5*(e2)$) -- ++($-2*(e3)$) -- ++($-2*4*(e2)$) -- ++($2*4*(e3)$) -- ++($-2*(e2)$) -- cycle;
\draw ($(O)+(-\d,\d)-2*(e3)$) -- ++($2*4*(e3)$) -- ++($-2*(e1)$) -- ++($-2*3*(e3)$) -- ++($2*6*(e1)$) -- ++($-2*(e3)$) -- cycle;
\foreach \x/\y/\z in {3/-1/0,-1/2/0}
    \filldraw[precornfill] ($(O)+2*\x*(e1)+2*\y*(e2)+2*\z*(e3)$) + (-\d,-\d) rectangle ++(\d,\d);
\foreach \x/\y/\z in {3/0/0,-1/0/2}
    \filldraw[fill=black!6, draw = black] ($(O)+(-\d,\d)+2*\x*(e1)+2*\y*(e2)+2*\z*(e3)$) -- ++($2*(e1)$) -- ++($-2*(e3)$) -- ++($-2*(e1)$) -- cycle;
\foreach \x/\y/\z in {0/2/0,0/-1/2}
    \filldraw[fill=black!17, draw = black] ($(O)+(-\d,\d)+2*\x*(e1)+2*\y*(e2)+2*\z*(e3)$) -- ++($2*(e2)$) -- ++($-2*(e3)$) -- ++($-2*(e2)$) -- cycle;
\foreach \x/\y/\z in {4/-1/0,-1/3/0,-1/-1/3}
    \filldraw[cornfill] ($(O)+2*\x*(e1)+2*\y*(e2)+2*\z*(e3)$) + (-\d,-\d) rectangle ++(\d,\d);
\foreach \x/\y/\z in {4/0/0,-1/0/3,-1/4/0}
    \filldraw[fill=black!25, draw = black] ($(O)+(-\d,\d)+2*\x*(e1)+2*\y*(e2)+2*\z*(e3)$) -- ++($2*(e1)$) -- ++($-2*(e3)$) -- ++($-2*(e1)$) -- cycle;
\foreach \x/\y/\z in {0/3/0,0/-1/3,5/-1/0}
    \filldraw[fill=black!37, draw = black] ($(O)+(-\d,\d)+2*\x*(e1)+2*\y*(e2)+2*\z*(e3)$) -- ++($2*(e2)$) -- ++($-2*(e3)$) -- ++($-2*(e2)$) -- cycle;
\end{tikzpicture}
\caption{A Young diagram with ``rays'' in Example~\ref{exYDsegm}.}
\label{YDraysfig}
\end{figure}
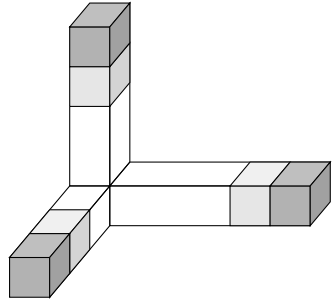

According to Proposition~\ref{propYD}, we have
\[f_d = \frac{(-1)^{d-1}}{d} \sum_{i: \, d_i = d} b_i z_i^d.\]
If $f_d$ is square-free, then $X$ is normal by Corollary~\ref{cornorm}. Otherwise, $f_d$ and $\frac{\pa f_d}{\pa z_1} = (-1)^{d-1}b_1z_1^{d-1}$ have a common divisor~$z_1$. It follows that if $X$ is non-normal, then $d_i = d$ only for $i = 1$, $d \ge 2$, $f_d = \frac{(-1)^{d-1}}{d}b_1z_1^d$, and $\widetilde f_d = \frac{(-1)^{d-1}}{d}b_1z_1^{d-1}$.

Let us check whether $\widetilde f_d$ and $f_{d-1}$ are coprime, i.e., $f_{d-1}$ is divisible by $z_1$. According to Proposition~\ref{propYD}, we have
\[f_{d-1} = \frac{(-1)^{d-2}}{d-1}\sum_{i:\,d_i = d-1} b_iz_i^{d-1} \;+\; (-1)^{d-2}b_1z_{2e_1}z_1^{d-2}.\]
It is divisible by $z_1$ if and only if there are no $d_i = d-1$ and $d \ge 3$.

\smallskip

We conclude that $X$ is normal if and only if $d_1 = 2$ or $d_2 \ge d_1-1$.
\end{example}

\begin{example}
\label{exYDsegm}
As a particular case of the previous example, we see that for a ``segment'' Young diagram of length $d+1$, i.e., for the Gorenstein local algebra $\KK[x]/(x^{d+1})$ with the subspace $U = \langle x, x^2, \ldots, x^{d-1}\rangle$, the corresponding hypersurface $X$ is non-normal for $d \ge 3$; see Figure~\ref{YDsegmfig}.
\begin{figure}[h]
\begin{tikzpicture}[x=0.75pt,y=0.75pt,yscale=-1,xscale=1]
\def\d{15}
\coordinate (e1) at (\d,0);
\coordinate (e2) at (0,-\d);
\foreach \x in {0, 1, 2, 4}
    \draw ($2*\x*(e1)$) + (-\d,-\d) rectangle ++(\d,\d);
\draw ($2*0*(e1)$) node{\scriptsize $1^{\phantom{1}}\!\!$};
\draw ($2*1*(e1)$) node{\scriptsize $x^{\phantom{1}}\!\!$};
\draw ($2*2*(e1)$) node{\scriptsize $x^2$};
\draw ($2*3*(e1)$) node{\scriptsize $\ldots$};
\draw ($2*4*(e1)$) node{\scriptsize $x^{d-2}$};
\filldraw[precornfill] ($2*5*(e1)$) + (-\d,-\d) rectangle ++(\d,\d);
\draw ($2*5*(e1)$) node{\scriptsize $x^{d-1}$};
\filldraw[cornfill] ($2*6*(e1)$) + (-\d,-\d) rectangle ++(\d,\d);
\draw ($2*6*(e1)$) node{\scriptsize $x^d$};
\end{tikzpicture}
\caption{A ``segment'' Young diagram in Example~\ref{exYDsegm}.}
\label{YDsegmfig}
\end{figure}
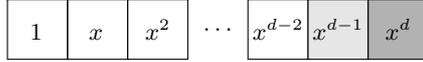
\end{example}

\begin{example}
The Young diagram on Figure~\ref{YD1fig} corresponds to a non-normal hypersurface~$X$: for corner cells $|(3,1)| = 4$, $|(1,2)| = 3$, so for $i=1$ the conditions of Proposition~\ref{prYD} hold.
\end{example}

\begin{example}
Consider the Young diagram $\Lambda$ that is the parallelepiped of size
\[(1+d_1) \times \ldots \times (1+d_k), \;\; d_i \ge 1,\]
and the corresponding hypersurface~$X$. The diagram~$\Lambda$ has a unique $\mu = (d_1, \ldots, d_k) \in \corn(\Lambda)$. Notice that $(A_{\Lambda,B}, U_{\Lambda,B}) = (\widetilde{A_{\Lambda,B}}, \widetilde{U_{\Lambda,B}})$ since the system of linear equations~\eqref{eqYDprecorn2} is $z_{\mu - e_i} = 0$, $1 \le i \le k$. If there is $d_i \ge 3$, then $X$ is non-normal by Proposition~\ref{prYD}. Otherwise, $X$ is normal since any irreducible divisor of $f_d = \frac{(-1)^{d-1}}{d} z^\mu$ is some variable~$z_i$, $1 \le i \le k$, and at least one monomial in the polynomial
\[f_{d-1} = (-1)^{d-2}\sum_{\substack{|\nu|=2\\ \nu \lecur \mu}} c_{\mu-\nu}b_\mu z_\nu z^{\mu-\nu}\]
is not divisible by $z_i$: for $\nu = 2e_i$ if $d_i = 2$ or for
$\nu = e_i + e_j$ if $d_i = 1$. The exceptional case for these arguments is the one-dimensional parallelepiped with $d_1 = 1$, which corresponds to a point in $\PP^1$. We conclude that $X$ is normal if and only if all $d_i$-s are not greater than~$2$.
\end{example}

\section{Hypersurfaces of maximal degree}
\label{hmd}

It is known that if a hypersurface $X\subseteq\PP^n$ admits an additive action, then its degree is at most $n$; see~\cite[Corollary~5.2]{AS}. Moreover, it follows from~\cite[Theorem~5.1]{AS} that such a hypersurface comes from an $H$-pair $(A,U)$, where $A=\KK[x]/(x^{n+1})$ and $U$ is a hyperplane in the maximal ideal $\mm=(x)$ that does not contain $x^n$.

\begin{proposition}
\label{pp}
For any $n\ge 2$, a hypersurface $X\subseteq\PP^n$ of degree $n$ that admits an additive action is unique up to automorphism of $\PP^n$.
\end{proposition}

\begin{proof}
In the notation introduced above, it suffices to prove that the subspace $U$ can be sent to the subspace $U_0:=\langle x,x^2,\ldots,x^{n-1}\rangle$ by an automorphism of the algebra $A$.
The set $C$ of all possible subspaces $U$ can be interpreted as an open subset in $\PP(\mm^*)$ isomorphic to the affine space $\AA^{n-1}$. Automorphisms of the algebra $A$ of the form
$$
x\mapsto x+a_2x^2+\ldots+a_nx^n, \quad a_2,\ldots,a_n\in\KK
$$
form an $(n-1)$-dimensional unipotent linear algebraic group $G$. This group acts naturally on $\PP(\mm^*)$ and preserves the subset $C$. We can write any $z\in\mm$ as
$z=z_1x+z_2x^2+\ldots+z_nx^n$ and consider $z_1,\ldots,z_n$ as a basis of $\mm^*$.

\begin{lemma}
\label{ll}
The stabilizer of the form $z_n$ in the group $G$ is trivial.
\end{lemma}

\begin{proof}
Assume that an automorphism $x\mapsto x+a_2x^2+\ldots+a_nx^n$ fixes $z_n$ and $a_k$ is the first non-zero coefficient. Applying this automorphism to
$z_1x+z_2x^2+\ldots+z_nx^n$, we see that $z_n$ goes to $z_n+(n-k+1)a_kz_{n-k+1}+\sum_{j=1}^{n-k}b_jz_j$ with some $b_j\in\KK$, a contradiction.
\end{proof}

The line $\KK z_n$ is a point in $C$. By Lemma~\ref{ll}, the $G$-orbit of this point is $(n-1)$-dimensional. By~\cite[Section~1.3]{PV}, any orbit of a unipotent group acting on an affine variety is closed. We conclude that the group $G$ acts on $C$ transitively, and the subspace $U$ can be sent to the subspace $U_0$ by an appropriate automorphism from $G$.
\end{proof}

Note that the hypersurface from Proposition~\ref{pp} is the hypersurface from Example~\ref{exYDsegm}. Let us take a closer look at this remarkable hypersurface. We already know that it is normal if and only if $n\le 2$. Let us write down explicitly the equation of $X$. By~\cite[Theorem~2.14]{AZa}, it has the form
$$
z_0^n\pi\left(\ln\Bigl(1+\frac{z}{z_0}\Bigr)\right)=\pi\left(\sum_{k=1}^n \frac{(-1)^{k-1}}{k}z_0^{n-k} (z_1x+\ldots+z_nx^n)^k\right)=0,
$$
where $\pi\colon \mm\to\mm/U_0$ is the projection. This gives the equation
$$
\sum_{k=1}^n \frac{(-1)^{k-1}}{k}z_0^{n-k}\sum_{j_1+\ldots+j_k=n} c_jz_{j_1}\ldots z_{j_k}=0
$$
For example, with $n=5$ we have
$$
z_0^4z_5-z_0^3z_1z_4-z_0^3z_2z_3+z_0^2z_1^2z_3+z_0^2z_1z_2^2-z_0z_1^3z_2+\frac{1}{5}z_1^5=0.
$$
The boundary $X_0$ is given in $X$ by $z_0=0$, so it is the subspace $z_0=z_1=0$ in $\PP^n$. The singular locus $\Xsing$
is empty if $n=2$ and coincides with $X_0$ if $n\ge 3$.


\section*{Statements and Declarations}

This work was supported by the Russian Science Foundation grant 23-21-00472.

\section*{Conflict of Interest}

The authors have no relevant financial or non-financial interests to disclose.

\section*{Data Availability}

Data sharing not applicable to this article as no datasets were generated or analysed during the current study.

\end{document}